\documentclass[reqno]{amsart} 
\usepackage{amssymb, dsfont, fancyhdr, graphicx, color, tabularx, mathtools, hyperref,tikz,comment}

\usepackage{geometry}
\usepackage[shortlabels]{enumitem} 
\usepackage[all]{xy} 
\usepackage{ytableau}
\usepackage[capitalize,nameinlink,noabbrev,nosort]{cleveref} 
\crefname{equation}{}{}
\usepackage{cite}

\def\Imm{\operatorname{Imm}}
\def\sgn{\operatorname{sgn}}

\newcommand{\CC}{{\mathbb C}}

\newcommand{\mc}[1]{\mathcal{#1}}

\newcommand{\Comment}[1]{{\color{blue} \sf ($\clubsuit$ #1 $\clubsuit$)}}

\newcommand{\Mat}{\text{Mat}}
\newcommand{\res}[2]{#1|_{#2}} 
\newcommand{\gr}[2]{\Gamma[#1, #2]} 
\newcommand{\gro}[1]{\Gamma(#1)} 
\newcommand{\wire}[2]{(#1,#2) }
\newcommand{\lrangle}[1]{\langle #1 \rangle}
\newcommand{\spn}[1]{\sigma(#1)}

\newcommand{\bbox}[1]{\mathbf{B}_{#1}}

\makeatletter
\newcommand{\specialcell}[1]{\ifmeasuring@#1\else\omit$\displaystyle#1$\ignorespaces\fi}
\makeatother

\newtheorem{prop}{Proposition}[section] 
\newtheorem{cor}[prop]{Corollary} 
\newtheorem{lem}[prop]{Lemma}
\newtheorem{conj}[prop]{Conjecture}
\newtheorem{thm}[prop]{Theorem}

\theoremstyle{definition}
\newtheorem{defn}[prop]{Definition}
\newtheorem{rmk}[prop]{Remark}
\newtheorem{ex}[prop]{Example}
\newtheorem{question}[prop]{Question}

\numberwithin{equation}{section} 

\title{1324- and 2143-avoiding Kazhdan-Lusztig immanants and $k$-positivity}
\author{Sunita Chepuri$^*$}\thanks{$^*$University of Minnesota, 127 Vincent Hall, 206 Church St. SE, Minneapolis, MN 55455, chepu003@umn.edu}
\author{Melissa Sherman-Bennett$^\dagger$}{\thanks{$^\dagger$University of California at Berkeley, Evans Hall, Berkeley, CA, m\_shermanbennett@berkeley.edu}}

\begin{document}

\maketitle
\begin{abstract}
    Immanants are functions on square matrices generalizing the determinant and permanent. Kazhdan-Lusztig immanants, which are indexed by permutations, involve $q=1$ specializations of Type A Kazhdan-Lusztig polynomials, and were defined by Rhoades and Skandera in \cite{RS}. Using results of \cite{HaiConj} and \cite{StemTNN}, Rhoades and Skandera showed that Kazhdan-Lusztig immanants are nonnegative on matrices whose minors are nonnegative. We investigate which Kazhdan-Lusztig immanants are positive on $k$-positive matrices (matrices whose minors of size $k \times k$ and smaller are positive). The Kazhdan-Lusztig immanant indexed by $v$ is positive on $k$-positive matrices if $v$ avoids 1324 and 2143 and for all non-inversions $i<j$ of $v$, either $j-i \leq k$ or $v_j-v_i \leq k$. Our main tool is Lewis Carroll's identity.
\end{abstract}

\section{Introduction}\label{sec:intro}
Given a function $f:S_n \to \CC$, the \emph{immanant} associated to $f$, $\Imm_f:\Mat_{n \times n}(\CC) \to \CC$, is the function

\begin{equation}
\Imm_f(M):= \sum_{w \in S_n} f(w) ~m_{1, w(1)} \cdots m_{n, w(n)}.
\end{equation}

Well-studied examples include the determinant, where $f(w)=(-1)^{\ell(w)}$, the permanent, where $f(w)=1$, and more generally \emph{character immanants}, where $f$ is an irreducible character of $S_n$.

In this paper, we establish the positivity of certain immanants evaluated on \emph{$k$-positive} matrices; that is, matrices whose minors of size at most $k$ are positive. Interest in positivity properties of immanants goes back to the early 1990s. Goulden and Jackson \cite{GJ} conjectured, and Greene \cite{Greene} later proved, that character immanants of Jacobi-Trudi matrices are monomial-positive. Jacobi-Trudi matrices refer to the matrices appearing in the famous Jacobi-Trudi identity for skew-Schur functions; they have entries that are homogeneous symmetric functions and determinants equal to skew-Schur functions. Greene's result was followed by a number of positivity conjectures by Stembridge \cite{StemConj}, including two that were proved shortly thereafter: Haiman showed that character immanants of generalized Jacobi-Trudi matrices (repeated rows and columns allowed) are Schur-positive \cite{HaiConj} and Stembridge \cite{StemTNN} showed that character immanants are nonnegative on \emph{totally nonnegative} matrices, matrices with nonnegative minors. 

%
    

In \cite[Question 2.9]{StemConj}, Stembridge also asks about the nonnegativity of certain immanants evaluated on \emph{$k$-nonnegative matrices}, matrices whose minors of size at most $k$ are nonnegative. (Stembridge's methods in \cite{StemTNN} do not give insight into this question, as it relies on a factorization of totally nonnegative matrices which does not exist for all $k$-nonnegative matrices.) 

Here, we investigate a variant of Stembridge's question, restricting our attention to $k$-positive matrices and \emph{Kazhdan-Lusztig immanants}, which were defined by Rhoades and Skandera \cite{RS}. 

\begin{defn}
	Let $v \in S_n$. The \emph{Kazhdan-Lusztig immanant} $\Imm_v: \Mat_{n \times n}(\CC) \to \CC$ is given by 
	\begin{equation} \label{eq:immFormula1}
	\Imm_v(M):= \sum_{w \in S_n} (-1)^{\ell(w)-\ell(v)} P_{w_0w, w_0v}(1) ~m_{1, w_1} \cdots m_{n, w_n}
	\end{equation}
	where $P_{x, y}(q)$ is the Kazhdan-Lusztig polynomial associated to $x,y \in S_n$, $w_0 \in S_n$ is the longest permutation, and we write permutations $w=w_1w_2\dots w_n$ in one-line notation. (For the definition of $P_{x, y}(q)$ and their basic properties, see e.g. \cite{BB}.)
\end{defn}

For example, letting $e$ denote the identity permutation, $\Imm_{e}(M)=\det M$ and $\Imm_{w_0}(M)=m_{n, 1}m_{n-1, 2} \cdots m_{1, n}$. 

Using results of \cite{ HaiConj, StemTNN}, Rhoades and Skandera \cite{RS} show that Kazhdan-Lusztig immanants are nonnegative on totally nonnegative matrices, and are Schur-positive on generalized Jacobi-Trudi matrices. Further, they show that character immanants are nonnegative linear combinations of Kazhdan-Lusztig immanants, so from the perspective of positivity, Kazhdan-Lusztig immanants are the more fundamental object to study.

We will call an immanant \emph{$k$-positive} if it is positive on all $k$-positive matrices. We are interested in the following question.

\begin{question} \label{quest:KLImmkPos}
	Let $0<k<n$ be an integer. For which $v \in S_n$ is $\Imm_v(M)$ $k$-positive?
\end{question}

Notice that $\Imm_e(M)=\det M$ is $k$-positive only for $k=n$. On the other hand, $\Imm_{w_0}$ is $k$-positive for all $k$, since it is positive as long as the entries (i.e. the $1 \times 1$ minors) of $M$ are positive. So, the answer to \cref{quest:KLImmkPos} is a nonempty proper subset of $S_n$.

Pylyavskyy \cite{Pyl} conjectured the following relationship between $\Imm_v(M)$ being $k$-positive and $v$ avoiding certain patterns (see \cref{defn:patternAvoidance}).



\begin{conj}[\hspace{1pt}\cite{Pyl}] \label{conj:pasha}
    Let $0<k<n$ be an integer and let $v \in S_n$ avoid the pattern $12\cdots (k+1)$. 
Then $\Imm_v(M) $ is $k$-positive. 
\end{conj}

Our main result is the following description of some $k$-positive Kazhdan-Lusztig immanants, in the spirit of Pylyavskyy's conjecture. 




\begin{thm}  \label{thm:123avoiding}
	Let $v \in S_n$ be 1324-, 2143-avoiding and suppose that for all $i<j$ with $v_i<v_j$, we have $j-i \leq k$ or $v_j-v_i \leq k$. Then $\Imm_v(M)$ is $k$-positive.
\end{thm}




To prove \cref{thm:123avoiding}, we first find a determinantal expression for $\Imm_v(M)$ when $v$ avoids 1324 and 2143. We then use Lewis Carroll's identity (also known as the Desnanot-Jacobi identity) to obtain sign information. We also show that \cref{thm:123avoiding} supports \cref{conj:pasha} (see \cref{prop:inc-pattern-avoiding}).

Before giving the proof of \cref{thm:123avoiding}, we point out that it can be rephrased in terms of Lusztig's \emph{dual canonical basis} for the coordinate ring of $GL_n(\CC)$. Indeed, Skandera~\cite{Skan} proved that every dual canonical basis element is, up to a power of $\det^{-1}$, a Kazhdan-Lusztig immanant evaluated on a matrix of indeterminates with repeated rows and columns. The minors are dual canonical basis elements, so \cref{thm:123avoiding} shows that the positivity of certain dual canonical basis elements (minors of size at most $k$) guarantees the positivity of others ($\Imm_v(N)$ where $N$ is a matrix of indeterminates and $v$ satisfies the hypotheses of the theorem). In fact, the positivity of minors of size at most $k$ guarantees the positivity of a broader class of dual canonical basis elements, which will be the subject of a forthcoming paper.

The paper is organized as follows.  Section~\ref{sec:defns} gives our conventions and the necessary definitions on permutations. In Section~\ref{sec: preliminaries}, we obtain a determinantal formula for $\Imm_v(M)$ when $v$ avoids 1324 and 2143.   Section~\ref{sec:mainthm} is the proof of our main result, Theorem~\ref{thm:123avoiding}.   In Section~\ref{sec:pattern-avoidance} we consider the condition on $v$ from Theorem~\ref{thm:123avoiding}: that for all $i<j$ with $v_i<v_j$ we have $j-i\leq k$ or $v_j-v_i\leq k$. We discuss how this condition relates to pattern avoidance and show that our main theorem supports Pylyavskyy's conjecture. In \cref{sec:connectionsToCluster}, we give some additional motivation and context for \cref{quest:KLImmkPos}; we state a more general version of \cref{quest:KLImmkPos} for arbitrary semisimple Lie groups and discuss connections to cluster algebras. Finally, Sections~\ref{sec:proof1} and~\ref{sec:proof2} provide proofs of technical lemmas used in Section~\ref{sec:mainthm}.

\section{Definitions}~\label{sec:defns}

For integers $i\leq j$, let $[i, j]:=\{i, i+1, \dots, j-1, j\}$. We abbreviate $[1, n]$ as $[n]$. For $v \in S_n$, we write $v_i$ or $v(i)$ for the image of $i$ under $v$. We use the notation $<$ for both the usual order on $[n]$ and the Bruhat order on $S_n$; it is clear from context which is meant. To discuss inversions or non-inversion of a permutation $v$, we'll write $\lrangle{i, j}$ to avoid confusion with a matrix index or point in the plane. In the notation $\lrangle{i, j}$, we always assume $i<j$.

We are concerned with two notions of positivity, one for matrices and one for immanants.

\begin{defn}
Let $k\geq 1$. A matrix $M \in \Mat_{n \times n}(\CC)$ is \emph{$k$-positive} if all minors of size at most $k$ are positive. 

An immanant $\Imm_f: \Mat_{n \times n}(\CC) \to \CC$ is \emph{$k$-positive} if it is positive on all $k$-positive matrices.
\end{defn}

Note that $k$-positive matrices have positive $1 \times 1$ minors, i.e. entries, and so are real matrices. 

Our results on $k$-positivity of Kazhdan-Lusztig immanants involve pattern avoidance.

\begin{defn}\label{defn:patternAvoidance}
	Let $v \in S_n$, and let $w\in S_m$. Suppose $v=v_1\cdots v_n$ and $w=w_1 \cdots w_m$ in one-line notation. The pattern $w_1 \cdots w_m$ \emph{occurs} in $v$ if there exists $1\leq i_1< \dots <i_m\leq n$ such that $v_{i_1} \cdots v_{i_m}$ are in the same relative order as $w_1 \cdots w_m$. Additionally, $v$ \emph{avoids} the pattern $w_1 \cdots w_m$ if it does not occur in $v$.
\end{defn}

In the following section, we will show that certain immanants have a very simple determinantal formula, which involves the \emph{graph} of an interval.

\begin{defn}\label{defn:graph}
For $v \in S_n$, the \emph{graph} of $v$, denoted $\gro{v}$, refers to its graph as a function. That is, $\gro{v}:=\{(1, v_1), \dots, (n, v_n)\}$.
For $v, w \in S_n$, the graph of the Bruhat interval $[v, w]$ is the subset of $[n]^2$ defined as $\gr{v}{w}:=\{(i, u_i): u \in [v, w], i=1, \dots, n\}$.
\end{defn} 

We think of an element $(i,j) \in \gr{v}{w}$ as a point in row $i$ and column $j$ of an $n \times n$ grid, indexed so that row indices increase going down and column indices increase going right (see \cref{ex:graph}). A \emph{square} or \emph{square region} in $\gr{v}{w}$ is a subset of $\gr{v}{w}$ which forms a square when drawn in the grid.

The following example illustrates the above concepts, as well as \cref{thm:123avoiding}.

\begin{ex}\label{ex:graph}
Consider $v=2413$ in $S_4$. We have $[v, w_0]=\{2413, 4213, 3412, 2431, 4312, 4231, 3421\}$, and so $\gr{v}{w_0}$ is as follows:
\begin{center}
   \includegraphics[height=0.15\textheight]{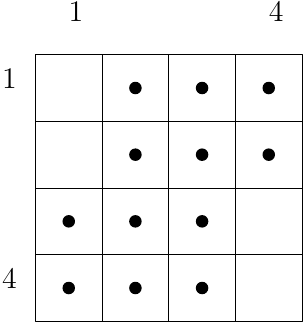}
\end{center}

Notice that $v$ avoids the patterns $1324$ and $2143$. Its non-inversions are $\lrangle{1, 2}, \lrangle{1, 4}, \lrangle{3, 4}$. For each of these non-inversions $\lrangle{i, j}$, we have that $j-i \leq 2$ or $v_j-v_i \leq 2$. So, \cref{thm:123avoiding} guarantees that
\begin{align*} \Imm_v(M)&= m_{12}m_{24}m_{31}m_{43} -m_{14}m_{22}m_{31}m_{43}- m_{13}m_{24}m_{31}m_{42}+m_{14}m_{23}m_{31}m_{42}\\
&-m_{12}m_{24}m_{33}m_{41}+m_{14}m_{22}m_{33}m_{41}+m_{13}m_{24}m_{32}m_{41}-m_{14}m_{23}m_{32}m_{41}
\end{align*}
is positive on all 2-positive $4 \times 4$ matrices.
\end{ex}

\section{Determinantal formulas for 4231- and 3412-avoiding Kazhdan-Lusztig immanants} \label{sec: preliminaries}

We first note that \cref{eq:immFormula1} has a much simpler form when $v$ is $1324$- and $2143$-avoiding. In general, computing $P_{x, y}(q)$, or even $P_{x, y}(1)$, is quite difficult, and there are no explicit combinatorial formulas for arbitrary $x, y$ (see e.g. \cite{BW} for formulas in special cases). As a result, Kazhdan--Lusztig immanants are also difficult to compute. However, when $v$ avoids $1324$ and $2143$, we can write a simple formula for them (\cref{cor:immDet}).

By \cite{KLPoincare}, $P_{x, y}(q)$ is the Poincar\'e polynomial of the local intersection cohomology of the Schubert variety indexed by $y$ at any point in the Schubert variety indexed by $x$; by \cite{LSsmoothSchub}, the Schubert variety indexed by $y$ is smooth precisely when $y$ avoids $4231$ and $3412$. These results imply that $P_{x, y}(q)=1$ for $y$ avoiding 4231 and 3412. Together with the fact that $P_{x, y}(q)=0$ for $x \nleq y$ in the Bruhat order, this gives the following lemma.

\begin{lem} \label{lem:immSimple}
	Let $v\in S_n$ be $1324$- and $2143$-avoiding. Then 
	\begin{equation} \label{eq:immSimple} \Imm_v(M)= (-1)^{\ell(v)} \sum_{w \geq v} (-1)^{\ell(w)} ~m_{1, w(1)} \cdots m_{n, w(n)}.
	\end{equation}
\end{lem}

The coefficients in the formula in \cref{lem:immSimple} suggest a strategy for analyzing $\Imm_v(M)$ for $v\in S_n$ avoiding $1324$ and $2143$: find some matrix $N$ such that $\det(N)=\pm \Imm_v(M)$. If such a matrix $N$ exists, the sign of $\Imm_v(M)$ is the sign of some determinant, which we have tools (e.g. Lewis Carroll's identity) to analyze. The most straightforward candidate for $N$ is a matrix obtained from $M$ by replacing some entries with $0$.

\begin{defn}
	Let $Q \subseteq[n]^2$ and let $M=(m_{ij})$ be in $\Mat_{n \times n}(\CC)$. The \emph{restriction of }$M$ \emph{to } $Q$, denoted $\res{M}{Q}$, is the matrix with entries 
	\[n_{ij}= \begin{cases}
	m_{ij} & \text{if } (i, j) \in Q\\
	0 & \text{else}.
	\end{cases}
	\]
\end{defn}

For a fixed $v \in S_n$ that avoids $1324$ and $2143$, suppose there exists $Q \subseteq[n]^2$ such that $\Imm_v(M)=\pm \det \res{M}{Q}$. Given the terms appearing in \cref{eq:immSimple}, $Q$ must contain $\gro{w}$ for all $w$ in $[v, w_0]$, and so must contain $\gr{v}{w_0}$. In fact, the minimal choice of $Q$ suffices. Before proving this, we give a characterization of $\gr{v}{w_0}$ as a subset of $[n]^2$.

\begin{defn}
Let $v \in S_n$ and $(i, j) \in [n]^2 \setminus \gro{v}$. Then $(i, j)$ is \emph{sandwiched} by a non-inversion (respectively, inversion) $\lrangle{k, l}$ if $k \leq i \leq l$ and $v_k \leq j \leq v_l$ (respectively, $v_k \geq j \geq v_l$). We also say $\lrangle{k, l}$ \emph{sandwiches} $(i, j)$.
\end{defn}

That is to say, $(i, j)$ is sandwiched by $\lrangle{k, l}$ if and only if, in the plane, $(i, j)$ lies inside the rectangle with opposite corners $(k, v_k)$ and $(l, v_l)$.

\begin{lem} \label{lem:graphCharacterization}
Let $v\in S_n$. Then $\gr{v}{w_0}=\gro{v} \cup \{(i, j): (i, j) \text{ is sandwiched by a non-inversion of } v\}$.
\end{lem}

\begin{figure} 
	\centering
	\includegraphics[height=1.5in]{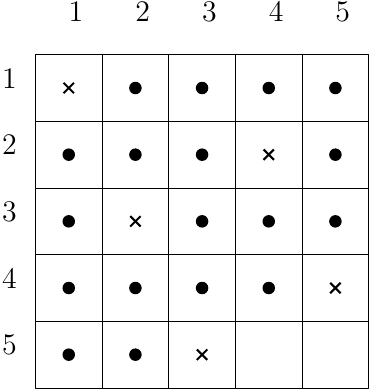}
	\caption{\label{fig:graph} An example of $\gr{v}{w_0}$ for $v=1 4 2 5 3$. Crosses mark the positions $(i, v_i)$, and dots mark all other elements of $\gr{v}{w_0}$.}
\end{figure}

\begin{proof}
Clearly $(i, v_i) \in \gr{v}{w_0}$ for all $i$, so suppose $(i, j)$ is sandwiched by a non-inversion $\lrangle{k, l}$ of $v$. We will produce a permutation $w>v$ sending $i$ to $j$, which shows that $(i, j) \in \gr{v}{w_0}$. Let $a=v^{-1}(j)$. If $a, i$ form a non-inversion of $v$, then $v\cdot (a~i)>v$ (where $(a~i)$ is the transposition sending $a$ to $i$ and vice versa) and $v(a~i)$ sends $i$ to $j$. If $a, i$ is an inversion, then $v \cdot (i~k) \cdot(i~a)>v$ if $i<a$ and $v\cdot (i~l) \cdot (i~a)>v$ if $i>a$. These permutations send $i$ to $j$, so we are done.

To show that the above description gives all elements of $ \gr{v}{w_0}$, suppose that $\lrangle{a, b}$ is a non-inversion of $v$ such that $\ell(v(a~b))=\ell(v)+1$. The graph of $v \cdot (a~ b)$ can be obtained from $\gro{v}$ by applying the following move: look at the rectangle bounded by $(a, v_a)$ and $(b, v_b)$, and replace $(a, v_a)$ and $(b, v_b)$ with the other corners of the rectangle, $(a, v_b)$ and $(b, v_a)$. Notice that $(a, v_b)$ and $(b, v_a)$ are sandwiched by the non-inversion $\lrangle{a, b}$. Further, if $(i, j)$ is sandwiched by a non-inversion of $v \cdot (a~b)$, then it is also sandwiched by a non-inversion of $v$. 
 Thus repeating this move produces graphs whose points are sandwiched by some non-inversion of $v$. Since for arbitrary $u>v$, the graph of $u$ can be obtained from that of $v$ by a sequence of these moves, we are done.
\end{proof}
Note that \cref{lem:graphCharacterization} implies $\gr{v}{w_0}$ is always a skew-shape. 

We are now ready to prove the following proposition, which follows from work of Sj\"{o}strand \cite{Sjo}. 

\begin{prop} \label{prop:fullIntervals}
	Let $v \in S_n$ avoid 1324, 24153, 31524, and 426153, and let $M\in \Mat_{n \times n}(\CC)$. Then 
	\[ \det(\res{M}{\gr{v}{w_0}})=\sum_{w \geq v} (-1)^{\ell(w)}~m_{1, w(1)} \cdots m_{n, w(n)}.
	\]
\end{prop}

\begin{proof}
Notice that by definition, 
\begin{equation}
\det(\res{M}{\gr{v}{w_0}})= \sum_{\substack{w \in S_n \\ \gro{w}\subseteq \gr{v}{w_0}}} (-1)^{\ell(w)}~m_{1, w(1)} \cdots m_{n, w(n)}.
\end{equation}

We show that this sum is in fact over $[v, w_0]$. \cref{lem:graphCharacterization} shows that for arbitrary $v \in S_n$, $\{w: \gro{w} \subseteq \gr{v}{w_0}\}$ are the permutations in what Sj\"{o}strand calls the ``left convex hull" of $v$. Applying \cite[Theorem 4]{Sjo}, we obtain that $\{w \in S_n: \gro{w} \subseteq \gr{v}{w_0}\}=[v, w_0]$ as desired.
\end{proof}

Combining \cref{lem:immSimple} and \cref{prop:fullIntervals}, we obtain an expression for certain immanants as determinants (up to sign).

\begin{cor} \label{cor:immDet} Let $v \in S_n$ avoid 1324 and 2143. Then

\begin{equation}\label{eq:immDet}
\Imm_v(M)=(-1)^{\ell(v)} \det(\res{M}{\gr{v}{w_0}}).
\end{equation}

\end{cor}

\begin{proof}
By \cref{lem:immSimple}, we have 
\begin{equation}\label{eq:immSimple2} \Imm_v(M)= (-1)^{\ell(v)} \sum_{w \geq v} (-1)^{\ell(w)} ~m_{1, w(1)} \cdots m_{n, w(n)}.
\end{equation}

Notice that $v$ also avoids 24153, 31524 and 426153, since the occurrence of any of those patterns would imply an occurrence of 2143. So by \cref{prop:fullIntervals}, the right-hand side of \cref{eq:immSimple2} is equal to $(-1)^{\ell(v)} \det(\res{M}{\gr{v}{w_0}})$.

\end{proof}

\begin{rmk}
\cref{cor:immDet} implies that $\Imm_v(M)$ can be efficiently computed when $v$ avoids 1324 and 2143, since it is a determinant.
\end{rmk}

We will use Lewis Carroll's identity to determine the sign of \cref{eq:immDet} in \cref{sec:mainthm}, using some results on the structure of $
\gr{v}{w_0}$.

\section{Dodgson Condensation} \label{sec:mainthm}

We are now ready to prove \cref{thm:123avoiding}.  We first note a useful lemma that will allow us to rephrase the theorem in terms of $\gr{v}{w_0}$ instead of non-inversions.

\begin{lem}\label{lem:sq-inversions}
Let $v\in S_n$. The graph $\gr{v}{w_0}$ has a square of size $k+1$ if and only if for some non-inversion $\lrangle{i, j}$ of $v$, we have $j-i \geq k$ and $v_j-v_i \geq k$.
\end{lem}

\begin{proof}
Let $Q$ be a square of size $k+1$ in $\gr{v}{w_0}$. Note that $Q$ is sandwiched by some non-inversion $\lrangle{i, j}$ of $v$. Indeed, let $(r_1, c_1)$ be the northwest corner of $Q$ and let $(r_2, c_2)$ be the southeast corner. Then $(r_1, c_1)$ is either sandwiched by a non-inversion $\lrangle{i, \ell}$ or is equal to $(i, v_i)$ for some $i$. In particular, $(i, v_i)$ is weakly northwest of $(r_1, c_1)$. Similarly, we can find $j$ so that $(j, v_j)$ is weakly southeast of $(r_2, c_2)$. So the non-inversion $\lrangle{i, j}$ sandwiches every point in $Q$. On the other hand, the rectangle with corners $(i, v_i), (j, v_i), (i, v_j), (j, v_j)$ uses $j-i+1$ columns and $v_j-v_i+1$ rows. So both $j-i+1$ and $v_j-v_i+1$ must be at least $k+1$, which shows one direction. The other direction is straightforward.
\end{proof}

Using this lemma, Theorem~\ref{thm:123avoiding} can be rephrased as follows.

\begin{thm}\label{thm:old}
Let $v \in S_n$ be 1324-, 2143-avoiding and suppose the largest square region in $\gr{v}{w_0}$ has size at most $k$. Then $\Imm_v(M)$ is $k$-positive. 
\end{thm}

The main technique we will use to prove \cref{thm:123avoiding} is application of the following:

\begin{prop}[Lewis Carroll's Identity]\label{prop:lc-id}
If $M$ is an $n\times n$ square matrix and $M_A^B$ is $M$ with the rows indexed by $A \subset [n]$ and columns indexed by $B \subset [n]$ removed, then 
$$\det(M)\det(M_{a,a'}^{b,b'})=\det(M_a^b)\det(M_{a'}^{b'})-\det(M_a^{b'})\det(M_{a'}^b),$$
where $1\leq a<a'\leq n$ and $1\leq b<b'\leq n$.
\end{prop}

\subsection{Young diagrams}

We begin by considering the cases where $\gr{v}{w_0}$ is a Young diagram or the complement of a Young diagram (using English notation). Recall that the \emph{Durfee square} of a Young diagram $\lambda$ is the largest square contained in $\lambda$.


\begin{prop} \label{prop:partitiondet}
Let $\lambda \subseteq n^n$ be a Young diagram with Durfee square of size $k$ and $\mu:=n^n/\lambda$. Let $M$ be an $n \times n$ $k$-positive matrix. Then 

\[(-1)^{|\mu|}\det (\res{M}{\lambda}) \geq 0
\]
and equality holds only if $(n, n-1, \dots, 1) \nsubseteq \lambda$.
\end{prop}

\begin{proof}
Let $A=\res{M}{\lambda}=\{a_{ij}\}$. For $\sigma \in S_n$, let $a_{\sigma}:=a_{1, \sigma(1)}\cdots a_{n, \sigma(n)}$. Note that if $\lambda_{n-j+1}<j$, then the last $j$ rows of $\lambda$ are contained in a $j \times (j-1)$ rectangle. There is no way to choose $j$ boxes in the last $j$ rows of $\lambda$ so that each box is in a different column and row. This means at least one of the matrix entries in $a_{\sigma}$ is zero. So $\det A=0$ if $(n, n-1, \dots, 1) \nsubseteq \lambda$.

Now, assume that $(n, n-1, \dots,1) \subseteq \lambda$. We proceed by induction on $n$ to show that $\det(A)$ has sign $(-1)^{|\mu|}$. The base cases for $n=1, 2$ are clear.

We would like to apply Lewis Carroll's identity to find the sign of $\det(A)$. Let $\lambda_I^J$ denote the Young diagram obtained from $\lambda$ by removing rows indexed by $I$ and columns indexed by $J$. Note that $A_I^J=\res{M_I^J}{\lambda_I^J}$. The submatrices $M_I^J$ are $k$-positive and the Durfee square of $\lambda_I^J$ is no bigger than the Durfee square of $\lambda$, so by the inductive hypothesis, we know the signs of $\det A_I^J$ for $|I|=|J|\geq 1$.

We will analyze the following Lewis Carroll identity:

\begin{equation}
\label{eqn:partitionid} 
\det(A) \det(A_{1, n}^{1, n})=\det(A_1^1)\det(A_n^n)-\det(A_1^n)\det(A_n^1).
\end{equation}

Note that $\lambda_{1, n}^{1, n}$ contains $(n-2, n-3, \dots, 1)$ and $\lambda_n^n, \lambda_1^n, \lambda_n^1$ contain $(n-1, n-2, \dots, 1)$, so the determinants of $A_{1, n}^{1, n}$, $A_n^n$, $A_1^n$, and $A_n^1$ are nonzero.

Suppose the last row of $\mu$ contains $r$ boxes and the last column contains $c$ (so that the union of the last column and row contains $r+c-1$ boxes). Note that $r, c<n$. Then $\det(A_{1, n}^{1, n})$ and $\det(A_n^n)$ have sign $(-1)^{|\mu|-(r+c-1)}$, $\det(A_1^n)$ has sign $(-1)^{|\mu|-c}$, and $\det(A_n^1)$ has sign $(-1)^{|\mu|-r}$. Notice that $\det(A_1^1)$ is either zero or it has sign $(-1)^{|\mu|}$, since $\mu_1^1=\mu$. In both of these cases, the right hand side of \cref{eqn:partitionid} is nonzero and has sign $(-1)^{-r-c+1}$; the left hand side has sign $\sgn(\det(A))\cdot(-1)^{|\mu|-r-c+1}$, which gives the proposition.

\end{proof}

\begin{cor}\label{cor:partitioncompdet}
Let $\mu \subseteq n^n$ be a Young diagram and let $\lambda:= n^n /\mu$. Suppose $\lambda$ has Durfee square of size $k$, and $M$ is a $k$-positive $n \times n$ matrix. Then 

\[ (-1)^{|\mu|}\det(\res{M}{\lambda}) \geq 0
\]
and equality holds if and only if $ (n^n/ (n-1, n-2, \dots, 1, 0) ) \subseteq \lambda$ (or equivalently, $\mu \subseteq (n-1, n-2, \dots, 1, 0)$).
\end{cor}

\begin{proof}
If we transpose $\res{M}{\lambda}$ across the antidiagonal, we obtain the scenario of \cref{prop:partitiondet}. Transposition across the antidiagonal is the same as reversing columns, taking transpose, and reversing columns again, which doesn't effect the sign of minors.
\end{proof}

\cref{prop:partitiondet} and \cref{cor:partitioncompdet} give us the following results about immanants.

\begin{cor}\label{cor:partitionimm}
Let $v \in S_n$ avoid 1324 and 2143. Suppose $\gr{v}{w_0}$ is a Young diagram $\lambda$ with Durfee square of size $k$. Then $\Imm_v(M)$ is $k$-positive.
\end{cor}

\begin{proof} Suppose $M$ is $k$-positive.  Note that $\gro{w_0} \subseteq \gr{v}{w_0}$ implies $\lambda$ contains the partition $(n, n-1, \dots, 1)$. Let $\mu=n^n/\lambda$.  By \cref{prop:partitiondet}, we know that $(-1)^{|\mu|}\det \res{M}{\gr{v}{w_0}}>0$.

In fact, there is a bijection between boxes of $\mu$ and inversions of $v$. If a box of $\mu$ is in row $r$ and column $c$, then $v(r)<c$ and $v^{-1}(c)<r$; otherwise, that box would be in $\gr{v}{w_0}$. This means exactly that $(v^{-1}(c),r)$ is an inversion. If $(a, b)$ was an inversion of $v$ and the box in row $b$ and column $v(a)$ was not in $\mu$, then for some $j$, the box in row $j$ and column $v(j)$ is southeast of the box in row $b$ and column $v(a)$. But then $1 ~v(a)~v(b)~v(j)$ would be an occurrence of the pattern 1324, a contradiction.

So, we know $(-1)^{|\mu|}\det \res{M}{\gr{v}{w_0}}=(-1)^{\ell(v)}\det \res{M}{\gr{v}{w_0}}>0$.  By \cref{cor:immDet}, this means $\Imm_v(M)>0$.
\end{proof}

\begin{cor}\label{cor:partitioncompimm}
Let $v \in S_n$ avoid 1324 and 2143. Suppose $\gr{v}{w_0}$ is $\lambda=n^n/\mu$ for some partition $\mu$ and the largest square in $\lambda$ is of size $k$. Then $\Imm_v(M)$ is $k$-positive.
\end{cor}

\begin{proof}
Suppose $M$ is $k$-positive.  Note that $\gro{w_0} \subseteq \gr{v}{w_0}$ implies $\lambda$ contains the partition $(n^n/(n-1, n-2, \dots, 1,0))$. By \cref{cor:partitioncompdet}, we know that $(-1)^{|\mu|}\det \res{M}{\gr{v}{w_0}}>0$.

As in the proof of \cref{cor:partitionimm}, there is a bijection between boxes of $\mu$ and inversions of $v$.  So, we know $(-1)^{|\mu|}\det \res{M}{\gr{v}{w_0}}=(-1)^{\ell(v)}\det \res{M}{\gr{v}{w_0}}>0$.  By \cref{cor:immDet}, this means $\Imm_v(M)>0$.
\end{proof}

\subsection{General case}

To prove \cref{thm:123avoiding}, we need to show that if $v$ avoids 1324 and 2143, then $\det\res{M}{\gr{v}{w_0}}$ has sign $(-1)^{\ell(v)}$ for $M$ satisfying the given positivity assumptions.

We first reduce to the case when $\gr{v}{w_0}$ is not block-antidiagonal. We will temporarily denote the longest element of $S_j$ by $w_{(j)}$.

\begin{lem} \label{lem:multipleblock}

Suppose $\gr{v}{w_0}$ is block-antidiagonal. Let $v_1 \in S_j$ and $v_2 \in S_{n-j}$ be permutations such that upper-right antidiagonal block of $\gr{v}{w_0}$ is equal to $\gr{v_1}{w_{(j)}}$ and the other antidiagonal block is equal to $\gr{v_2}{w_{(n-j)}}$. If $M$ is an $n\times n$ matrix, then

\[(-1)^{\ell(v)}\det \res{M}{\gr{v}{w_0}}=(-1)^{\ell(v_1)}\det \res{M_1}{\gr{v_1} {w_{(j)}}} \cdot (-1)^{\ell(v_2)}\det \res{M_2}{\gr{v_2} {w_{(n-j)}}}
\]
where $M_1$ (resp. $M_2$) is the square submatrix of $M$ using columns $n-j+1$ through $n$ and rows $1$ through $j$ (resp. columns $1$ through $n-j$ and rows $j+1$ through $n$).
\end{lem}

See \cref{fig:blockAntiDiagEx} for an example of $v_1$ and $v_2$.

\begin{proof}[Proof of \cref{lem:multipleblock}]

Notice that $v$ is the permutation

\[ v: i \mapsto 
\begin{cases} v_1(i)+ n-j& \text{ if } 1 \leq i \leq j\\
v_2(i-j) & \text{ if } j< i \leq n.
\end{cases}
\]


For a block-antidiagonal matrix $A$ with blocks $A_1, A_2$ of size $j$ and $n-j$, respectively, we have 
\begin{align*}\det(A)&=(-1)^{\ell(w_0)}\det(Aw_0)\\
&= (-1)^{\ell(w_0)}  \det(A_1 w_{(j)})~ \det (A_2 w_{(n-j)})\\
&= (-1)^{\ell(w_0)+ \ell(w_{(j)})+\ell(w_{(n-j})} \det(A_1) \det(A_2)\\
&=(-1)^{\binom{n}{2}+\binom{j}{2}+\binom{n-j}{2}} \det(A_1) \det(A_2).
\end{align*}

From the above description of $v$, $\ell(v)= j(n-j)+\ell(v_1)+\ell(v_2)$. This, together with the fact that $\binom{n}{2}=\binom{j}{2}+\binom{n-j}{2}+j(n-j)$ and the formula for the determinant of a block-antidiagonal matrix, gives the desired equality.
\end{proof}

\begin{figure}
    \centering
    \includegraphics[height=0.3\textheight]{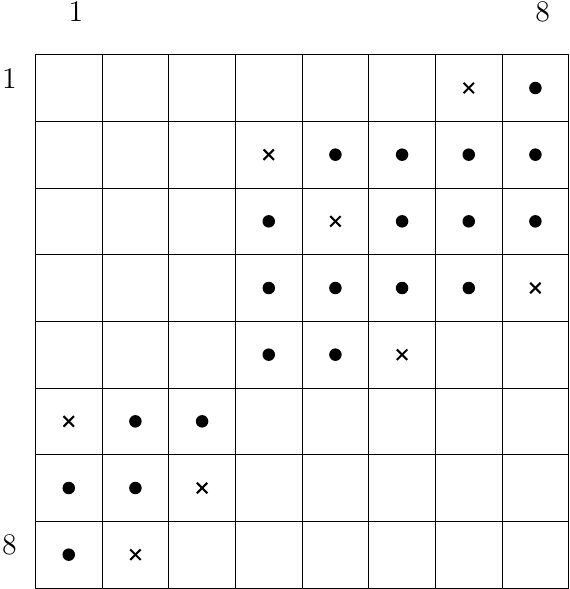}
    \caption{An example where $\gr{v}{w_0}$ is block-antidiagonal. Here, $v=74586132$. In the notation of \cref{lem:multipleblock}, $j=3$, $v_1=41253$, and $v_2=132$.}
    \label{fig:blockAntiDiagEx}
\end{figure}

From \cref{cor:immDet} and \cref{lem:multipleblock}, we have the following corollary. 

\begin{cor} \label{lem:ImmMultipleBlock}
Let $v \in S_n$ avoid 1324 and 2143, let $M$ be an $n\times n$ matrix, and let $v_i$, and $M_i$ be as in \cref{lem:multipleblock}. Then 
\[\Imm_v(M)=\Imm_{v_1}(M_1)\Imm_{v_2}(M_2).
\]
\end{cor}

We now introduce two propositions we will need for the proof of the general case.

\begin{defn}\label{defn:bounding-box-anti}
Let $v \in S_n$.  Define $\bbox{i, v_i}$ to be the square region of $[n]^2$ with corners $(i,v_i),\ (i, n-i+1),\ (n-v_i+1, v_i)$ and $(n-v_i+1,n-i+1)$.  In other words, $\bbox{i, v_i}$ is the square region of $[n]^2$ with a corner at $(i,v_i)$ and 2 corners on the antidiagonal of $[n]^2$. We say $\bbox{i, v_i}$ is a \emph{bounding box} of $\gr{v}{w_0}$ if there does not exist some $j$ such that $\bbox{i,v_i}\subsetneq \bbox{j,v_j}$. If $\bbox{i, v_i}$ is a bounding box of $\gr{v}{w_0}$, we call $(i, v_i)$ a \emph{spanning corner} of $\gr{v}{w_0}$. We denote the set of spanning corners of $\gr{v}{w_0}$ by $S$. (See \cref{fig:boundingBoxEx} for an example.)
\end{defn}

\begin{figure}
    \centering
    \includegraphics[height=0.35\textheight]{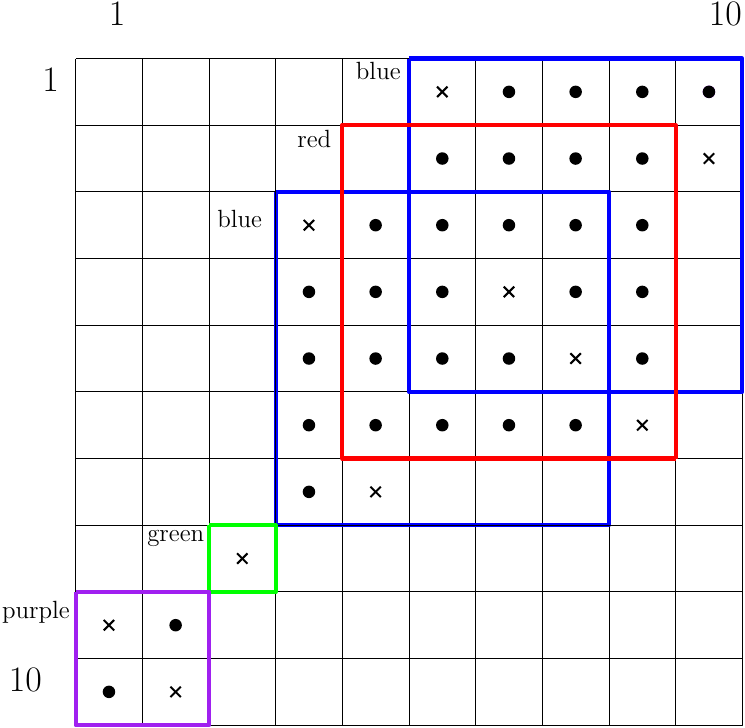}
    \caption{An example of $\gr{v}{w_0}$, with $v=6~10~4~7~8~9~3~1~2$. The bounding boxes are blue, red, blue, green, and purple, listed in the order of their northmost row. The spanning corners of $\gr{e}{v}$ are $(1, 6)$, $(3, 4)$, $(6, 9)$, $(8, 3)$, $(9, 1)$, and $(10, 2)$.}
    \label{fig:boundingBoxEx}
\end{figure}

\begin{rmk} \label{rmk:corners}
To justify the name ``spanning corners," notice that if $(i, v_i)$ is not sandwiched by any non-inversion of $v$, then $(i, v_i)$ is a corner of $\gr{v}{w_0}$ (i.e. there are either no elements of $\gr{v}{w_0}$ weakly northwest of $(i, v_i)$ or no elements of $\gr{v}{w_0}$ weakly southeast of $(i, v_i)$). Conversely, if $(i, v_i)$ is sandwiched by a non-inversion $\lrangle{k, l}$ of $v$, then $(i,v_i)$ is in the interior of $\gr{v}{w_0}$, and $\bbox{i, v_i} \subseteq \bbox{k, v_k}$. So all elements of $S$ are corners of $\gr{v}{w_0}$.
\end{rmk}

The name ``bounding boxes" comes from the following lemma.

\begin{lem} \label{lem:boundingboxes}
Let $v \in S_n$. Then
\[ \gr{v}{w_0} \subseteq \bigcup_{(i, v_i) \in S} \bbox{i, v_i}.
\]
\end{lem}

\begin{proof}

Let $R_{k, l}$ denote the rectangle with corners $(k, v_k), (l, v_l), (k, v_l),$ and $(l, v_k)$. If $\lrangle{k, l}$ is a non-inversion of $v$, then $R_{k, l}$ consists exactly of the points sandwiched by $\lrangle{k, l}$. So by \cref{lem:graphCharacterization}, we have that

\[
\gr{e}{w}=\bigcup_{\lrangle{k, l} \text{ non-inversion of } v} R_{k, l}.
\]

Notice that if $(i, w_i)$ is sandwiched by $\lrangle{k, l}$, then $ R_{i, l}$ and $R_{k, i}$ are contained in $R_{k, l}$. So to show the desired containment, it suffices to show that $R_{k, l}$ is contained in $\bigcup_{\wire{i}{v_i} \in S} \bbox{i, v_i}$ for all non-inversions $\lrangle{k, l}$ such that $(k, v_k)$ and $(l, v_l)$ are both corners of $\gr{v}{w_0}$.

So consider a non-inversion $\lrangle{k, l}$ where $\wire{k}{v_k}, \wire{l}{v_l}$ are corners. Working through the possible relative orders of $k, l, v_k, v_l$, one can see that $R_{k, l} \subseteq \bbox{k, v_k} \cup \bbox{l, v_l}$. Since $S$ consists of spanning corners, we have $\wire{a}{v_a}, \wire{b}{v_{b}} \in S$ such that $\bbox{a, v_a}$ and $\bbox{b, v_b}$ contain $\bbox{k, w_k}$ and $\bbox{l, w_{l}}$, respectively. So $R_{k, l} \subseteq \bbox{a, w_a} \cup \bbox{b, w_b}$, as desired. 
\end{proof}

We also color the bounding boxes.

\begin{defn}
 A bounding box $\bbox{i,v_i}$ is said to be \emph{red} if $(i,v_i)$ is below the antidiagonal, \emph{green} if $(i,v_i)$ is on the antidiagonal, and \emph{blue} if $(i,v_i)$ is above the antidiagonal.  In the case where $\bbox{i,v_i}$ is a bounding box and $\bbox{n-v_i+1,n-i+1}$ is also a bounding box, then $\bbox{i,v_i}=\bbox{n-v_i+1,n-i+1}$ is both red and blue, in which case we call it \emph{purple}. (See \cref{fig:boundingBoxEx} for an example.)
\end{defn}

\begin{prop} \label{prop:alternatingBoxesAnti}
Suppose $v \in S_n$ avoids 2143 and $w_0v$ is not contained in a maximal parabolic subgroup of $S_n$. Order the bounding boxes of $\gr{v}{w_0}$ by the row of the northwest corner. If $\gr{v}{w_0}$ has more than one bounding box, then they alternate between blue and red and there are no purple bounding boxes.
\end{prop}

\begin{ex} If $v=3 7 1 4 5 6$, then $\gr{v}{w_0}$ is the same as the upper right anti-diagonal block of \cref{fig:boundingBoxEx}. The bounding boxes of $v$, in the order of \cref{prop:alternatingBoxesAnti}, are $\bbox{1, 3}$, $\bbox{6, 5}$, and $\bbox{3, 1}$. The colors are blue, red, and blue, respectively.
\end{ex}

For the next proposition, we need some additional notation.

Define $\delta_i:[n] \setminus \{i\} \to [n-1]$ as

\[ 
\delta_i(j)= \begin{cases}
j, & j<i;\\
j-1, & j>i.
\end{cases}
\]

\begin{defn}
For $i, k \in [n]$ and $P \subseteq [n]^2$, let $P^k_i\subseteq[n-1]\times [n-1]$ be $P$ with row $i$ and column $k$ deleted. That is, $P^k_i=\{(\delta_i(r), \delta_k(c)): (r, c) \in P\}$.
\end{defn}

\begin{prop} \label{prop:deleteDotDetAnti}

Let $v\in S_n$ be 2143- and 1324-avoiding, and choose $i\in [n]$. Let $x \in S_{n-1}$ be the permutation $x: \delta_i(j) \mapsto \delta_{v_i}(v_j)$ (that is, $x$ is obtained from $v$ by deleting $v_i$ from $v$ in one-line notation and shifting the remaining numbers appropriately). If $(i, v_i)$ is not a spanning corner of $\gr{v}{w_0}$, then $\gr{x}{w_0}=\gr{v}{w_0}_i^{v_i}$. Moreover, for all $i$,
\begin{equation} \label{eqn:deleteDotDetAnti}
\det(\res{M}{\gr{x}{w_0}})=\det(\res{M}{\gr{v}{w_0}_{i}^{v_i}}).\end{equation}
\end{prop}

\begin{figure}
    \centering
    \includegraphics[height=0.4\textheight]{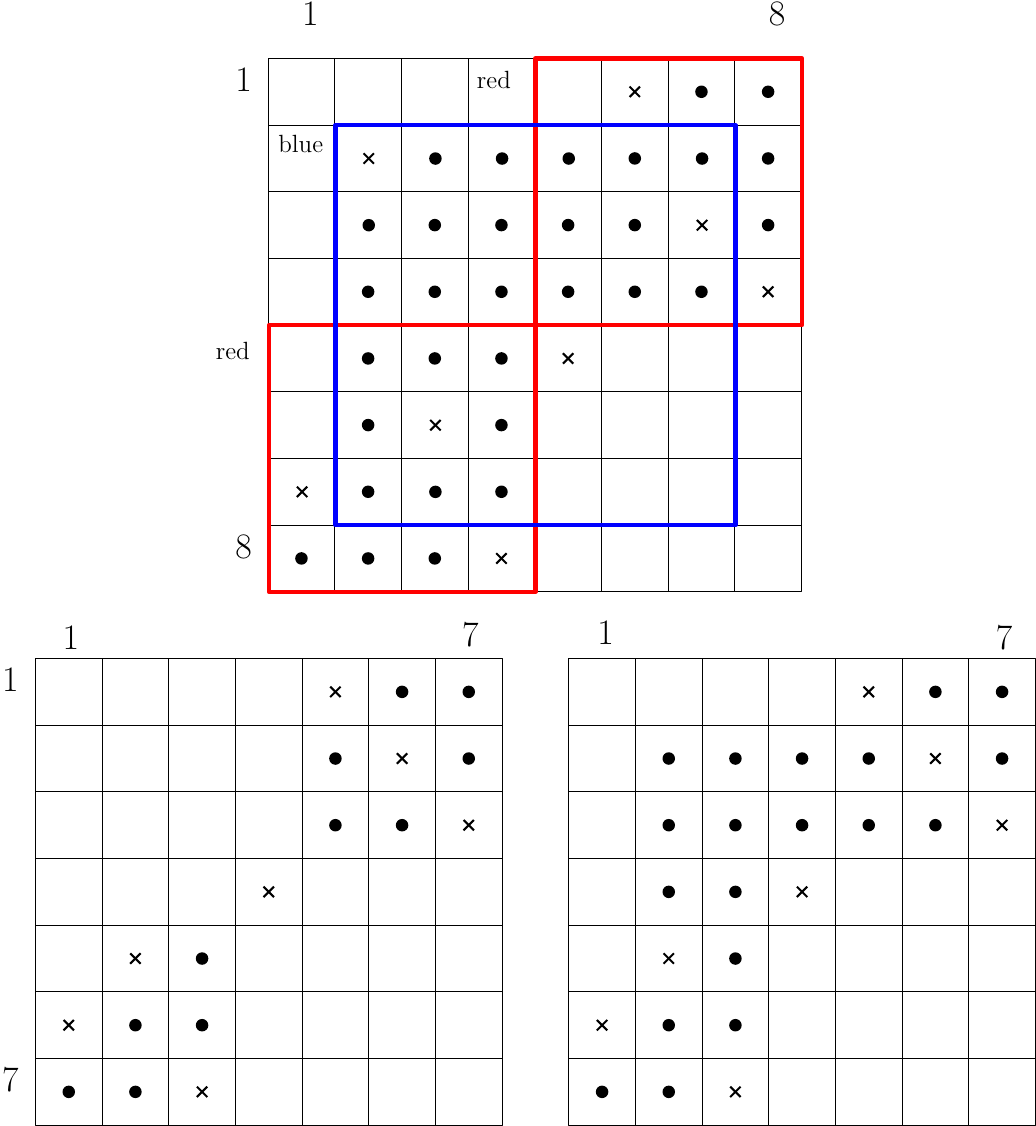}
    \caption{At the top, $\gr{v}{w_0}$ for $v=62785314$. On the bottom left, the region $\gr{x}{w_0}$, where $x=5674213$ is the permutation obtained by deleting 2 from the one-line notation of $v$ and shifting remaining numbers appropriately. On the bottom right, the region $\gr{v}{w_0}^2_3$. As you can see, the determinant of $\res{M}{\gr{x}{w_0}}$ is the same as the determinant of $\res{M}{\gr{v}{w_0}^2_3}$ for all $7 \times 7$ matrices $M$, illustrating \cref{prop:deleteDotDetAnti}. }
    \label{fig:deleteDotDetEx}
\end{figure}

The proofs of these propositions are quite technical and appear below \cref{sec:proof1,sec:proof2} respectively.

\begin{figure}
    \centering
    \includegraphics[height=0.45\textheight]{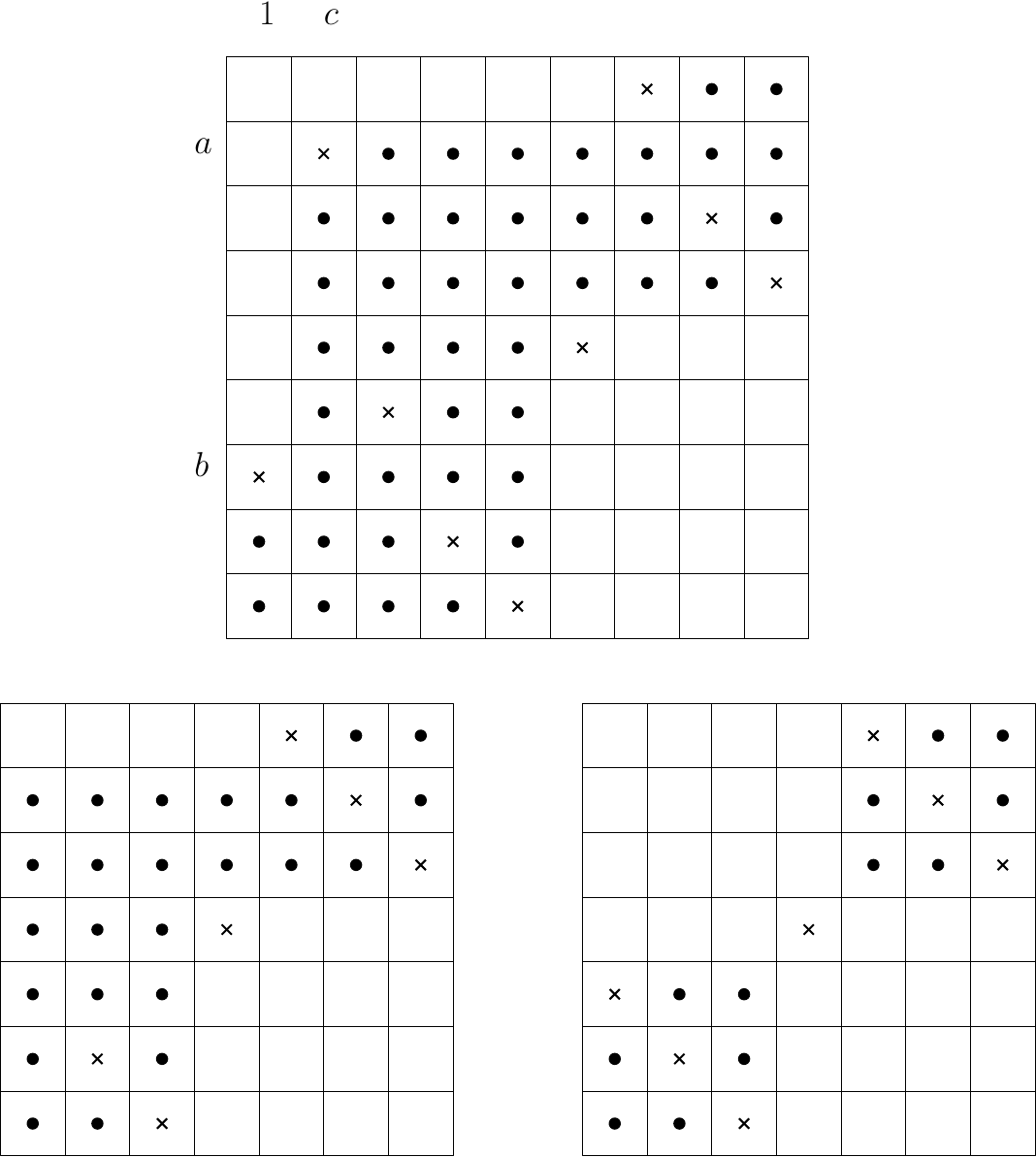}
    \caption{An illustration of (2) in the proof of \cref{thm:detSign}. Let $v=62785314$; $\gr{v}{w_0}$ is shown with bounding boxes at the top of \cref{fig:deleteDotDetEx}. In this case, $a=c=2$, $b=7$, $w=728963145$ and $y=5674123$. On the top is $\gr{w}{w_0}$. On the bottom left is $\gr{v}{w_0}^1_a$, which is equal to $\gr{w}{w_0}^{1, c}_{a, b}$. On the bottom right is $\gr{y}{w_0}$.}
    \label{fig:weirdTermDodgson}
\end{figure}

\begin{thm} \label{thm:detSign}
Let $v \in S_n$ avoid 1324 and 2143 and let $k$ be the size of the largest square in $\gr{v}{w_0}$. For $M$ $k$-positive, $(-1)^{\ell(v)} \det \res{M}{\gr{v}{w_0}}>0$.
\end{thm}

\begin{proof}

We proceed by induction on $n$; the base case $n=2$ is clear. 

Now, let $n>2$. If $\gr{v}{w_0}$ is a partition or a complement of a partition (that is, it has a single bounding box), we are done by \cref{cor:partitionimm} or \cref{cor:partitioncompimm}. If it is block-antidiagonal (that is, $w_0v$ is contained in some parabolic subgroup), then we are done by \cref{lem:ImmMultipleBlock}. So we may assume that $v$ has at least 2 bounding boxes and that adjacent bounding boxes have nonempty intersection (where bounding boxes are ordered as usual by the row of their northeast corner). Indeed, if adjacent bounding boxes have empty intersection, then the fact that $\gr{v}{w_0}$ is contained in the union of bounding boxes (\cref{lem:boundingboxes}) implies that $\gr{v}{w_0}$ is block-antidiagonal.

Because $v$ avoids 1324 and 2143, the final two bounding boxes of $\gr{v}{w_0}$ are of opposite color by \cref{prop:alternatingBoxesAnti}. Without loss of generality, we assume the final box is red and the second to last box is blue. (Otherwise, we can consider the antidiagonal transpose of $M$ restricted to $\gr{w_0 v^{-1} w_0}{w_0}$, which has the same determinant.) 

This means the final box is $\bbox{n, v_n}$, and the second to last box is $\bbox{a, v_a}$ for some $a<n$ with $1<v_a<v_n$. We analyze the sign of $\det \res{M}{\gr{v}{w_0}}$ using Lewis Carroll's identity on rows $a, b:=v^{-1}(1)$ and columns $1, c:=v_a$. Note that $a<b$ and $1<c$.

We will consider the determinants appearing in Lewis Carroll's identity one by one, and show that they are of the form $\det(\res{N}{\gr{u}{w_0}})$ for some permutation $u$. Note that $(\res{M}{\gr{v}{w_0}})^i_j=\res{M^i_j}{\gr{v}{w_0}^i_j}$.

\begin{enumerate}
    \item Consider $(\res{M}{\gr{v}{w_0}})^{1, c}_{a, b}$. Next we show that the determinant of this matrix is equal to $\det(\res{M^{1, c}_{a, b}}{\gr{x}{w_0}})$ where $x$ is the permutation obtained from $v$ by deleting 1 and $v_a$ from $v$ in one-line notation and shifting remaining values to obtain a permutation of $[n-2]$. Indeed, $(b, 1)$ is a corner of $\gr{v}{w_0}$ but not a spanning corner. By \cref{prop:deleteDotDetAnti}, we have that $\gr{v}{w_0}^1_b=\gr{u}{w_0}$, where $u$ is obtained from $v$ by deleting 1 from $v$ in one-line notation and shifting appropriately. So we have 
    
    \begin{align*}(\res{M}{\gr{v}{w_0}})^{1, c}_{a, b}&= ((\res{M}{\gr{v}{w_0}})^1_ b)^{c-1}_{a}\\
    &=(\res{M^1_b}{\gr{u}{w_0}})^{c-1}_{a}\\
    &= \res{M^{1, c}_{a, b}}{\gr{u}{w_0}^{c-1}_a}.
    \end{align*}
    
    Note that $u_a=c-1$ and deleting $c-1$ from the one-line notation of $u$ gives $x$. So taking determinants and applying \cref{prop:deleteDotDetAnti} gives the desired equality.
    
    Note that $\ell(x)=\ell(v)-a-b-c+4$. Indeed, 1 is involved in exactly $b-1$ inversions of $v$. Because there are no $(j, v_j)$ northwest of $(a, c)$, $a$ is involved in exactly $a+c-2$ inversions of $v$: each of the $c-1$ columns to the left of column $c$ contains a dot $(j, v_j)$ southwest of $(a, c)$, and each of the $a-1$ rows above $(a, c)$ contains a dot $(j, v_j)$ northeast of $(a, c)$. We have counted the inversion $\lrangle{1, a}$ twice, so deleting $1$ and $c$ from $v$ deletes $(b-1)+(a+c-2)-1$ inversions.
    
    \item  Consider $(\res{M}{\gr{v}{w_0}})^1_a$. The determinant of this matrix is equal to the determinant of $\res{M^1_a}{\gr{y}{w_0}}$, where $y$ is obtained from $v$ by adding $v_n+1$ to the end of $v$ in one-line notation and shifting appropriately to get $w$, then deleting 1 and $v_a$ from $w$ and shifting values to obtain a permutation of $[n-1]$. 
    
    To see this, first note that there are no pairs $(i, v_i)$ with $n>i>b$ and $v_i>v_n$; such a pair would mean that $v_a 1 v_i v_n$ form a 2143 pattern, which is impossible. There are also no pairs $(i, v_i)$ with $i<a$ and $v_i<v_n$. Indeed, because $(a, v_a)$ is a corner, we would have to have $v_a<v_i<v_n$. This means we would have a red bounding box $\bbox{j, v_j}$ following $\bbox{a, v_a}$ with $a<j$ and $v_n<v_j$. Then $v_i v_a v_b v_n$ would form a 2143 pattern. 
    This implies that the number of elements of $\gr{v}{w_0}$ in row $b$ and in row $n$ are the same. Similarly, the number of elements of $\gr{v}{w_0}$ in column $c$ and column $v_n$ are the same. So to obtain $\gr{w}{w_0}$ from $\gr{v}{w_0}$, add a copy of row $b$ of $\gr{v}{w_0}$ below the last row, a copy of column $c$ to the right of the $v_n$th row, and a dot in position $(n+1, v_{n+1})$. It follows that $\gr{w}{w_0}^{1, c}_{a, b}=\gr{v}{w_0}^1_a$.
    
    One can check that $w$ (and thus $y$) avoid 2143 and 1324. A similar argument to \emph{(1)} shows that the determinant of $\res{M}{\gr{w}{w_0}^{1, c}_{a, b}}$ is equal to the determinant of $\res{M^{1}_{a}}{\gr{y}{w_0}}$.
    
    Note that $\ell(y)=\ell(v)-a-b-c+4 + n -v_n$.
    
    \item Consider $(\res{M}{\gr{v}{w_0}})^c_b$. The determinant of this matrix is equal to the determinant of $\res{M^c_b}{\gr{z}{w_0}}$, where $z$ is obtained from $v$ by deleting $v_n$ from $v$ in one-line notation and shifting as necessary. This follows from the fact that $\gr{v}{w_0}$ has the same number of dots in rows $b$ and $n$, and in columns $c$ and $v_n$ (proved in \emph{(2)}). So $\gr{v}{w_0}^c_b=\gr{v}{w_0}^{v_n}_n$. The claim follows from \cref{prop:deleteDotDetAnti} applied to $\res{M}{\gr{v}{w_0}^{v_n}_n}$.
    
    Note that $\ell(z)=\ell(v)-(n-v_n)$.

    \item Consider $(\res{M}{\gr{v}{w_0}})^1_b$. By \cref{prop:deleteDotDetAnti}, the determinant of this matrix is equal to the determinant of $\res{M^1_b}{\gr{p}{w_0}}$, where $p$ is obtained from $v$ by deleting 1 from $v$ in one-line notation and shifting appropriately.
    
    Note that $\ell(p)=\ell(v)-b+1$.
    
    \item Consider $(\res{M}{\gr{v}{w_0}})^c_a$. By \cref{prop:deleteDotDetAnti}, the determinant of this matrix is equal to the determinant of $\res{M^c_a}{\gr{q}{w_0}}$, where $q$ is obtained from $v$ by deleting $v_a$ from $v$ in one-line notation and shifting appropriately.
    
    Note that $\ell(q)=\ell(v)-c-a+2$.
\end{enumerate}

Notice that the permutations $x, y, z, p, q$ listed above avoid 2143 and 1324. By induction, the determinant of $\res{M_I^J}{\gr{u}{w_0}}$ has sign $\ell(u)$, and, in particular, is nonzero, so we know the signs of each determinant involved in Lewis Carroll's identity besides $\det M$. Dividing both sides of Lewis Carroll's identity by $\det (\res{M}{\gr{v}{w_0}})^{1, c}_{a, b}$, we see that both terms on the right-hand side have sign $\ell(v)$. Thus, the right-hand side is nonzero and has sign $\ell(v)$, which completes the proof.

\end{proof}

Taking this theorem with \cref{cor:immDet}, we can now prove \cref{thm:old}, and thus \cref{thm:123avoiding}.

\begin{proof}[Proof of \cref{thm:old}]
By \cref{cor:immDet}, $$\Imm_v(M)= (-1){\ell(v)}\det \res{M}{\gr{v}{w_0}}.$$

Let $k'\leq k$ be the size of the largest square in $\gr{v}{w_0}$. By \cref{thm:detSign}, for $M$ $k'$-positive, the right hand side of this expression is positive. Any $k$-positive matrix is also $k'$-positive, so we are done. 
\end{proof}

\begin{proof}[Proof of Theorem~\ref{thm:123avoiding}]
Let $v \in S_n$ is 1324-, 2143-avoiding.  If for all $i<j$ with $v_i<v_j$, we have $j-i \leq k$ or $v_j-v_i \leq k$, then by Lemma~\ref{lem:sq-inversions}, $\gr{v}{w_0}$ does not contain a square of size $k+1$. Thus, by Theorem~\ref{thm:old}, $\Imm_v(M)$ is $k$-positive.
\end{proof}

The methods above give an elementary proof of the following, which is a special case of the results of Rhoades--Skandera on Kazhdan-Lusztig immanants evaluated on totally nonnegative matrices. 

\begin{cor}
    Let $v \in S_n$ avoid 1324 and 2143. Then $\Imm_v(M) >0$ for all $n$-positive (that is, ``totally positive") $n \times n$ matrices $M$.
\end{cor}
\begin{proof}
If $v\in S_n$ then $\gr{v}{w_0}$ has no square of size $n+1$.  So, by Theorem~\ref{thm:123avoiding}, if $v$ avoids 1324 and 2143 and $M$ is a totally positive $n\times n$ matrix, then $\Imm_v(M) >0$.
\end{proof}

\section{Pattern Avoidance Conditions}\label{sec:pattern-avoidance}

In this section, we investigate the relationship between the assumptions of Theorem~\ref{thm:123avoiding} and pattern avoidance. First, we show that \cref{thm:123avoiding} supports \cref{conj:pasha}; that is, all permutations satisfying the assumptions of \cref{thm:123avoiding} also satisfy the assumptions of \cref{conj:pasha}.

\begin{prop}\label{prop:inc-pattern-avoiding}
If $v\in S_n$ and for all $i<j$ with $v_i<v_j$ we have $j-i \leq k$ or $v_j-v_i \leq k$, then $v$ avoids $12...(k+1)$.
\end{prop}

\begin{proof}
Suppose $i_1<i_2<...<i_{k+1}$ and $v_{i_1}<v_{i_2}<...<v_{i_{k+1}}$.  In other words, $v_{i_1},v_{i_2},...,v_{i_{k+1}}$ is an occurrence of the pattern $12...(k+1)$.  Let $R$ be the rectangle with corners at $(i_1,v_{i_1})$, $(i_1, v_{i_{k+1}})$, $(i_{k+1}, v_{i_{k+1}})$, and $(i_{k+1}, v_{i_1})$.  Notice that $R$ is at least of size $(k+1)\times(k+1)$. For all $(r, c) \in R$, $(r, c)$ is sandwiched by $\lrangle{i_1, i_{k+1}}$, a non-inversion.  By Lemma~\ref{lem:graphCharacterization}, this means $(r, c)$ is in $\gr{v}{w_0}$.  So, all of $R$ is is in $\gr{v}{w_0}$ and there is a square of size $k+1$ in $\gr{v}{w_0}$.  By \cref{lem:sq-inversions}, this means there is some non-inversion $\langle i,j\rangle$ where $j-i>k$ or $v_j-v_i>k$.
\end{proof}


Next, we consider pattern avoidance conditions for a permutation $v$ that are equivalent to the condition that $\gr{v}{w_0}$ has no square of size $k+1$, and thus equivalent to the contition that for all non-inversions $\langle i,j\rangle$ we have $j-i \leq k$ or $v_j-v_i \leq k$. That is, we give a way to phrase the assumptions of \cref{thm:123avoiding} entirely in terms of pattern avoidance.

\begin{prop}\label{prop:pattern-difference}
Let $v \in S_n$.  Then $\gr{v}{w_0}$ contains a square of size $k+1$ if and only if for all $k+1 \leq m \leq 2k$, all patterns $u_1 \dots u_m$ occurring in $v$ satisfy $u_{i+k}-u_i \neq k$ for all $i$.
\end{prop}

\begin{proof}
By Lemma~\ref{lem:sq-inversions}, $\gr{v}{w_0}$ contains a square of size $k+1$ if and only if, for some non-inversion $\lrangle{i, j}$, $j-i \geq k$ and $v_j-v_i \geq k$. We will show that the existence of such an inversion is equivalent to the statement about patterns in the proposition. 

Suppose for some non-inversion $\lrangle{i, j}$ of $v$, $j-i \geq k$ and $v_j-v_i \geq k$. Now, take 
\[L:= \{v_i +1, \dots, v_j-1\} \cap \{v_{i+1}, \dots,v_{j-1}\}.
\]
If $L$ has at least $k-1$ elements, let $S=T$ be any $(k-1)$-element subset of $L$. Otherwise, let $S$ be any $(k-1)$-element subset of $\{v_i +1, \dots, v_j-1\}$ containing $L$ and let $T$ be any $(k-1$-element subset of $\{v_{i+1}, \dots,v_{j-1}\}$ containing $L$. Notice that $S \cup T$ contains exactly $k-1$ numbers between $v_i$ and $v_j$ and exactly $k-1$ numbers that lie to the right of $v_i$ and to the left of $v_j$ in $v$. Now, we consider the pattern formed by $ S \cup T \cup \{v_i, v_j\}$ in $v$; say this pattern is $u_1 \dots u_m$. By construction $m$ is between $k+1$ and $2k$. Say $u_r$ corresponds to $v_i$. By construction, $u_{r+k}$ corresponds to $v_j$ and $u_{r+k}-u_r=k$.

Now, let $u_1 \dots u_m$ with $k+1 \leq m \leq 2m$ be a pattern such that $u_{r+k}-u_r=k$ for some $r$. If this pattern appears in $v$, let $v_i$ correspond to $u_r$ and $v_j$ correspond to $u_{r+k}$. Note that $\lrangle{i, j}$ is a non-inversion of $v$. To pass from $u$ to $v$, we insert additional numbers and increase existing ones, so $j-i \geq k$ and $v_j - v_i \geq k$.
\end{proof}

\begin{prop}\label{prop:sq-box}
If $v$ is 1324-, 2143-avoiding and there is a square of size $r$ in $\gr{v}{w_0}$, then there is a square of size $r$ that has a spanning corner of $\gr{v}{w_0}$ as either its most northwestern or most southeastern box.
\end{prop}

\begin{proof}
Let $Q$ be a square of size $r$ in $\gr{v}{w_0}$.  Without loss of generality, $Q$ lies on or above the anti-diagonal (the argument is analogous if the northeast corner of $Q$ lies on or below the anti-diagonal).  Then by Lemma~\ref{lem:graphCharacterization} the northwest corner of $Q$ is either $(i,v_i)$ or it is sandwiched by some non-inversion $\langle i,j\rangle$.  In either case, $\bbox{i,v_i}$ contains all of $Q$.  Since $\bbox{i,v_i}$ is either a bounding box or contained in a bounding box, $Q$ must be contained in a single bounding box.

Now we can assume $Q$ is contained within the bounding box $\bbox{i,v_i}$.  If $\bbox{i,v_i}$ is blue then there is some $(j,v_j)\in\gro{v}$ such that either $(j,v_j)$ is the southeast corner $Q$ or $\langle i,j\rangle$ is a non-inversion that sandwiches all of $Q$.  Consider the square of size $r$ with northwest corner $(i,v_i)$. Everything in this square is in $\gro{v}$ or is sandwiched by $\langle i,j\rangle$. Thus, by Lemma~\ref{lem:graphCharacterization}, this square is in $\gr{v}{w_0}$.  Similarly, if the bounding box is red then the square of size $r$ with southeast corner $(i,v_i)$ is in $\gr{v}{w_0}$.
\end{proof}

Taking these two propositions together, we can rewrite Theorem~\ref{thm:123avoiding} fully in terms of pattern avoidance conditions.

\begin{thm}\label{thm:patterns}
Let $v\in S_n$ avoid 1324, 2143, and all patterns $u$ of length $k+1\leq m\leq 2k$ where $u_1=1$ and $u_{k+1}=k+1$ or where $u_m=m$ and $u_{m-k}=m-k$. Then $\Imm_v$ is $k$-positive.
\end{thm}

\begin{proof}
Suppose $v\in S_n$ avoids 1324 and 2143 and $\gr{v}{w_0}$ contains a square of size $k+1$.  Then by Proposition~\ref{prop:sq-box}, $\gr{v}{w_0}$ contains a square of size $k+1$ that has a spanning corner as either its most northwestern or most southeastern box.  From the proof of Proposition~\ref{lem:sq-inversions}, this means there is a square of size $k+1$ in $\gr{v}{w_0}$ sandwiched by the non-inversion $\langle i,j\rangle$ with either $i$ or $j$ a spanning corner.
Let $u=u_1\dots u_m$ be the pattern constructed in Proposition~\ref{prop:pattern-difference} with $u_{\ell+k}-u_\ell=k$.  If $i$ is a spanning corner then $\ell=1$ and $u_\ell=1$ and if $j$ is a spanning corner then $\ell+k=m$ and $u_{\ell+k}=m$.

Thus, if $v$ avoids 1324 and 2143 as well as all patterns $u$ of length $k+1\leq m\leq 2k$ where $u_1=1$ and $u_{k+1}=k+1$ or where $u_m=m$ and $u_{m-k}=m-k$, then $\gr{v}{w_0}$ contains no square of size $k+1$.  By Theorem~\ref{thm:old}, $\Imm_v$ is $k$-positive.
\end{proof}

\begin{rmk} We note that some of the pattern avoidance conditions in \cref{thm:patterns} are repetitive, as some patterns listed contain others. For example, if $k=2$ the patterns listed include both 123 and 1324. So one can check that $v$ satisfies the hypotheses of \cref{thm:patterns} by checking it avoids a somewhat smaller list of patterns.
\end{rmk}

We get the following immediate corollary from Theorem~\ref{thm:patterns}:

\begin{cor}\label{cor:2-pos}
	Let $v \in S_n$ avoid 123, 2143, 1432, and 3214. Then $\Imm_v(M)$ is 2-positive.
\end{cor}

However, analogous statements for $k>2$ are difficult to state.  The larger $k$ is, the more patterns need to be avoided in order to mandate that $\gr{v}{w_0}$ has no square of size $k+1$.  We illustrate with the following example.

\begin{ex}
Let $v$ be 1324-, 2143-avoiding.  Then $\gr{v}{w_0}$ does not have a square of size 4, if and only if $v$ also avoids the following patterns: 1234, 15243, 15342, 12543, 13542, 32415, 42315, 32145, 42135, 165432, 156432, 165423, 156423, 543216, 543216, 453216, 543126, 453126.
\end{ex}

Due the number of patterns to be avoided, statements analogous to Corollary~\ref{cor:2-pos} for larger $k$ seem unlikely to be useful.

\section{Final remarks}\label{sec:connectionsToCluster}

In this section, we give some additional motivation and context for \cref{quest:KLImmkPos}. For an arbitrary reductive group $G$, Lusztig~\cite{LusTPGr} defined the totally positive part $G_{>0}$ and showed that elements of the dual canonical basis of $\mc{O}(G)$ are positive on $G_{>0}$. Fomin and Zelevinsky~\cite{FZTNNSemisimple} later showed that for semisimple groups, $G_{>0}$ is characterized by the positivity of generalized minors, which are dual canonical basis elements corresponding to the fundamental weights of $G$ and their images under Weyl group action. Note that the generalized minors are a finite subset of the (infinite) dual canonical basis, but their positivity guarantees the positivity of all other elements of the basis. 

In the case we are considering, $G=GL_n(\CC)$, $G_{>0}$ consists of the totally positive matrices and generalized minors are just ordinary minors. Skandera~\cite{Skan} showed that Kazhdan-Lusztig immanants are part of the dual canonical basis of $\mc{O}(GL_n(\CC))$, which gives another perspective on their positivity properties. (In fact, Skandera proved that every dual canonical basis element is, up to a power of $\det^{-1}$, a Kazhdan-Lusztig immanant evaluated on matrices with repeated rows and columns.) In light of these facts, \cref{quest:KLImmkPos} becomes a question of the following kind.

\begin{question} \label{quest:CanBasiskPos}
	Suppose some finite subset $S$ of the dual canonical basis is positive on $M \in G$. Which other elements of the dual canonical basis are positive on $M$? In particular, what if $S$ consists of the generalized minors corresponding to the first $k$ fundamental weights and their images under the Weyl group action?
\end{question}

These questions have a similar flavor to positivity tests arising from cluster algebras, which is different than the approach we take here. The coordinate ring of $GL_n$ is a cluster algebra, with some clusters given by double wiring diagrams \cite{CA3}. The minors are cluster variables. If we restrict our attention to the minors of size at most $k$ in the clusters for $GL_n$, we obtain a number of sub-cluster algebras, investigated by the first author in \cite{kPosTests}. The cluster monomials in those sub-algebras will be positive on $k$-positive matrices. Thus, we could ask the following to find more permutations $v$ where $\Imm_v$ is $k$-positive.

\begin{question}
    When is $\Imm_v$ a cluster monomial in $\mathcal{O}(GL_n(\CC))$?  When is $\Imm_v$ a cluster monomial in a $k$-positivity cluster sub-algebra from \cite{kPosTests}?
\end{question}

Interestingly, the Kazhdan-Lusztig immanants of 123-, 2143-, 1423-, and 3214-avoiding permutations do appear in sub-cluster algebras of this kind for $k=2$. In general, however, it is not known if $\Imm_v$ is a cluster variable in the cluster structure on $GL_n$, or in the sub-cluster algebras using only minors of size at most $k$. It is conjectured that cluster monomials form a (proper) subset of the dual canonical basis, so the cluster algebra approach would at best provide a partial answer to \cref{quest:CanBasiskPos}.

\section{Proof of Proposition~\ref{prop:alternatingBoxesAnti}}\label{sec:proof1}

In order to simplify the proofs in the next two sections, we will consider the graphs of lower intervals $[e, w]$ rather than upper intervals $[v, w_0]$. As the next lemma shows, the two are closely related.

\begin{lem} \label{lem:reverseCols}
Let $v \in S_n$. Then $\gr{v}{w_0}=\{(i, w_0(j)): (i, j) \in \gr{e}{w_0v}\}$. In other words, $\gr{v}{w_0}$ can be obtained from $\gr{e}{w_0v}$ by reversing the columns of the $n \times n$ grid.
\end{lem}

\begin{proof}
This follows immediately from the fact that left multiplication by $w_0$ is an anti-automorphism of the Bruhat order.
\end{proof}

Since left-multiplication by $w_0$ takes non-inversions to inversions, we have an analogue of \cref{lem:graphCharacterization} for $\gr{e}{w}$.

\begin{lem} \label{lem:lowerIntervalGraphCharacterization}
Let $w \in S_n$. Then $\gr{e}{w}=\gro{w} \cup \{(i, j) \in [n]^2: (i, j) \text{ is sandwiched by an inversion of } w\}$.
\end{lem}

Left-multiplication by $w_0$ takes the anti-diagonal to the diagonal, so we also have an analogue of bounding boxes and \cref{lem:boundingboxes}. 

\begin{defn}
Let $w \in S_n$. Define $B_{i, w_i} \subseteq[n]^2$ to be the square region with corners $(i, i)$, $(i, w_i)$, $(w_i, i)$, $(w_i, w_i)$. We call $B_{i, w_i}$ a \emph{bounding box} of $\gr{e}{w}$ if it is not properly contained in any $B_{j, w_j}$. In this situation, we call $(i, w_i)$ a \emph{spanning corner} of $\gr{e}{w}$. We denote the set of spanning corners of $\gr{e}{w}$ by $\overline{S}$.
\end{defn}

\begin{lem} \label{lem:boundingboxesDiag}
Let $w \in S_n$. Then
\[\gr{e}{w} \subseteq \bigcup_{\wire{i}{w_i} \in \overline{S}} B_{i, w_i}.
\]
\end{lem}

\begin{defn}\label{defn:bounding-box-diag}
A bounding box $B_{i, w_i}$ is colored \emph{red} if $i > w_i$, \emph{green} if $i=w_i$, and \emph{blue} if $i<w_i$. If $w^{-1}(i)=w_i$ (so that both $(i, w_i)$ and $(w_i, i)$ are spanning corners), then the bounding box $B_{i, w_i}=B_{w_i, i}$ is both red and blue. If a bounding box is both red and blue, we also call it \emph{purple}.
\end{defn}

\begin{rmk}\label{rmk:bounding-boxes}
Note that if $w=w_0v$, then $\bbox{i,v_i}$ is a bounding box of $\gr{v}{w_0}$ if and only if $B_{i,w_i}$ is a bounding box of $\gr{e}{w}$.  Further, $\bbox{i, v_i}=\{(r, w_0(c)): (r, c) \in B_{i, w_i}\}$ and $\bbox{i, v_i}$ has the same color as $B_{i, w_i}$.
\end{rmk}

We introduce one new piece of terminology.

\begin{defn}
Let $w \in S_n$. The \emph{span} of $(i, w_i)$, denoted by $\spn{i,w_i}$, is $[i, w_i]$ if $i\leq w_i$ and is $[w_i, i]$ otherwise. We say $(i, w_i)$ \emph{spans} $(j, w_j)$ if $\spn{i, w_i}$ contains $\spn{j, w_j}$ (equivalently, if $B_{i, w_i}$ contains $B_{j, w_j}$).
\end{defn}

Rather than proving \cref{prop:alternatingBoxesAnti} directly, we instead prove the following: 

\begin{prop} \label{prop:alternatingboxesDiag}
Suppose $w \in S_n$ avoids 3412 and is not contained in a maximal parabolic subgroup of $S_n$. Order the bounding boxes of $\gr{e}{w}$ by the row of the northwest corner. If $\gr{e}{w}$ has more than one bounding box, then they alternate between blue and red and there are no purple bounding boxes.
\end{prop}

The first step to proving this is analyzing the bounding boxes of $\gr{e}{w}$ when $w$ is a 321- and 3412-avoiding permutation. We'll need the following result of \cite{Ten} and a lemma.

\begin{prop}\cite[Theorem 5.3]{Ten} \label{prop:Tenner} Let $w \in S_n$. Then $w$ avoids 321 and 3412 if and only if for all $a \in [n-1]$, $s_a$ appears at most once in every (equivalently, any) reduced expression for $w$.

\end{prop}

\begin{lem} \label{lem:321smalltranspositions}
Suppose $w \in S_n$ avoids 321 and 3412. If for some $i \in [n]$, $w^{-1}(i)=w_i$, then $|i-w_i|=1$.
\end{lem}

\begin{proof}

Suppose $w^{-1}(i)=w_i$, and let $t$ be the transposition sending $i$ to $w_i$, $w_i$ to $i$, and fixing everything else. We can assume that $i<w_i$. We compare $w$ and $t$ in the Bruhat order by comparing $w[j]$ and $t[j]$ for $j \in [n]$. 

For two subsets $I=\{i_1< \dots < i_r\}$, $K=\{k_1< \dots < k_r\}$, we say $I \leq J$ if $i_j\leq k_j$ for all $j \in [r]$. For two permutations $v, w \in S_n$, we have $v \leq w$ if and only if for all $j \in [n]$, $v[j] \leq w[j]$ \cite[Theorem 2.6.1]{BB}.

For $j<i$ and $j\geq w_i$, $t[j]=[1, j]$, and so clearly $t[j]\leq w[j]$. For $i \leq j <w_i $, $t[j]=[1, i-1] \cup [i+1, j] \cup \{w_i\}$. Let $t[j]=\{a_1 < a_2 < \cdots < a_j\}$ and $w[j]=\{b_1 < b_2 < \cdots < b_j\}$. Notice that $a_j=w_i$ and $w_i \in w[j]$, so we definitely have that $a_j \leq b_j$. For the other inequalities, suppose $w[j] \cap [1, i]=\{b_1, \dots, b_r\}$. Since $i$ is not in $w[j]$, $r \leq i-1$. This implies that $b_k \leq a_k$ for $k \leq r$. It also implies that $a_{r+k} \leq i+k$ and $b_{r+k} \geq i+k$ for $r+k<j$, which establishes the remaining inequalities. Thus, $w \geq t$.

 Since $w \geq t$, every reduced expression for $w$ has a reduced expression for $t$ as a subexpression. Thus, \cref{prop:Tenner} implies $t$ must have a reduced expression in which each simple transposition appears once.

Now, $t$ has a reduced expression which is a palindrome; it is length $2c+1$ and contains at most $c+1$ simple transpositions. Every reduced expression of $t$ has the same length and contains the same set of simple transpositions. So if $c>0$, in each reduced expression, some simple transposition appears twice. We conclude that $\ell(t)=1$ and $|i -w_i|=1$.

\end{proof}

\begin{prop} \label{prop:alternatingboxes321}
Suppose $w \in S_n$ avoids 321 and 3412 and is not contained in any maximal parabolic subgroup. Order the bounding boxes of $\gr{e}{w}$ by the row of their northwest corner. Then no bounding boxes of $\gr{e}{w}$ are green or purple, and they alternate between red and blue.
\end{prop}

\begin{proof}

If a bounding box $B_{i, w_i}$ is green, then by definition $i=w_i$. The corner $(i, i)$ has maximal span, which implies there are no $(j, w_j)$ with $j <i$ and $w_j>i$. In other words, $w[i-1]=[i-1]$, which would contradict the assumption that $w$ is not contained in a maximal parabolic subgroup.

There are also no spanning corners of the form $(i, i+1)$. Indeed, if $(i, i+1)$ were a spanning corner, then there are no $(j, w_j)$ with $j\leq i-1$ and $w_j>i+1$ or with $j\geq i+1$ and $w_j\leq i-1$. The latter implies that $w_{i+1}=i$, which, together with the first inequality, implies $w[i-1]=[i-1]$, a contradiction. By \cref{lem:321smalltranspositions}, if a bounding box $B_{i, w_i}$ is purple, then $|i-w_i|=1$. This implies a spanning chord of the form $(k, k+1)$, which is impossible, so there are no purple bounding boxes.

So all bounding boxes are either blue or red. Suppose the bounding box $B_{i, w_i}$ is followed by $B_{j, w_j}$ in the ordering specified in the proposition. We suppose for the sake of contradiction that they are the same color. Without loss of generality, we may assume that they are blue, so $i<w_i$ and $j<w_j$. Otherwise, we consider $w^{-1}$ instead of $w$. (By \cref{prop:Tenner}, $w^{-1}$ is also 321- and 3412-avoiding. The span of $(i, w_i)$ and $(w_i, w^{-1}(w_i))$ are the same, so the bounding boxes of $\gr{e}{w^{-1}}$ are the same as the bounding boxes of $\gr{e}{w}$, but with opposite color.)

Since $B_{j, w_j}$ follows $B_{i, w_i}$, there are no pairs $(k, w_k)$ with $k<j$ and $w_i<w_k$. Indeed, such a pair is spanned by neither $(i, w_i)$ nor $(j, w_j)$, so its existence would imply the existence of a bounding box between $B_{i, w_i}$ and $B_{j, w_j}$ or enclosing one of them, both of which are contradictions. In other words, $w[j-1] \subseteq [w_i]$, so we must have $j-1 \leq w_i$. If $j=w_i+1$, then $w[w_i] \subseteq [w_i]$, a contradiction. So we have $i < j \leq w_i < w_j$. 

Now, consider the reduced expression for $w$ obtained by starting at the identity, moving $w_1$ to position 1 using right multiplication by $s_a$'s, then moving $w_2$ to position 2 (if it is not already there) using right multiplication, etc. Note that when $w_k$ is moved to position $k$, no numbers greater than $w_k$ have moved. Also, once $w_k$ is in position $k$, it never moves again. Suppose $w_{i-1}$ has just moved to position $i-1$. Because $(i, w_i)$ is a spanning corner, we have not moved any numbers larger than $w_i$. In other words, $k$ is currently in position $k$ for $k \geq w_i$; in particular, $w_i$ is in position $w_i$. Now, to move $w_i$ to position $i$, we must use the transpositions $s_{w_i-1}, s_{w_i -2}, \dots, s_{i+1}, s_i$ in that order. By \cref{prop:Tenner}, each simple transposition can only be used once in this reduced expression for $w$. Thus, these simple transpositions have not been used before we move $w_i$ to position $i$, so in fact $k$ is in position $k$ for $k >i$ before we move $w_i$ to position $i$.

Now we move $w_i$ to position $i$. Since $s_{w_i-1}, \dots, s_i$ will never be used again in the expression for $w$, we conclude that $w_{i+1}, \dots, w_{w_i-1}$ are already in positions $i+1, \dots, w_i -1$. Note also that the number currently in position $w_i$ is $w_i-1$, since $|i-w_i|>1$. 

Since $w_j>w_i$, $w_j$ is not yet in position $j$. This implies that $j \geq w_i$. We already had that $j \leq w_i$, so in fact they are equal. So after $w_i$ has moved to position $i$, $w_j$ is the next number not yet in the correct position. Recall that for $k >w_i$, $k$ is still in position $k$. So to move $w_j$ to position $w_i$, we use $s_{w_j}, s_{w_j-1}, \dots, s_j$ in that order (since $|j-w_j|>1$, $s_{j+1}$ is on this list of transpositions). Notice that the number in position $j$, which is $w_i-1$, moves to position $j+1$. We cannot use $s_j$ or $s_{j+1}$ again, so $w_i-1$ must be $w_{j+1}$. However, the pair $(j+1, w_i-1)$ satisfies $j+1>j$ and $w_i-1<w_i$, so it is spanned by neither $(i, w_i)$ nor $(j, w_j)$. Say $(j+1, w_i-1)$ is spanned by the spanning corner $(a, w_a)$. Then $(a, w_a)$ spans neither of $(i, w_i), (j, w_j)$, which implies $i < \min(a, w_a)<j<\max(a, w_a)<w_j$. This means exactly that the bounding box order is $B_{i, w_i}, B_{a, w_a}, B_{j, w_j}$, a contradiction.

\end{proof}

The next proposition allows us to extend \cref{prop:alternatingboxes321} to all permutations avoiding 3412.

Recall that $\delta_i:[n] \setminus \{i\} \to [n-1]$ is defined as

\[ 
\delta_i(j)= \begin{cases}
j & j<i\\
j-1 & j>i.
\end{cases}
\]

\begin{figure}
    \centering
    \includegraphics[height=0.2\textheight]{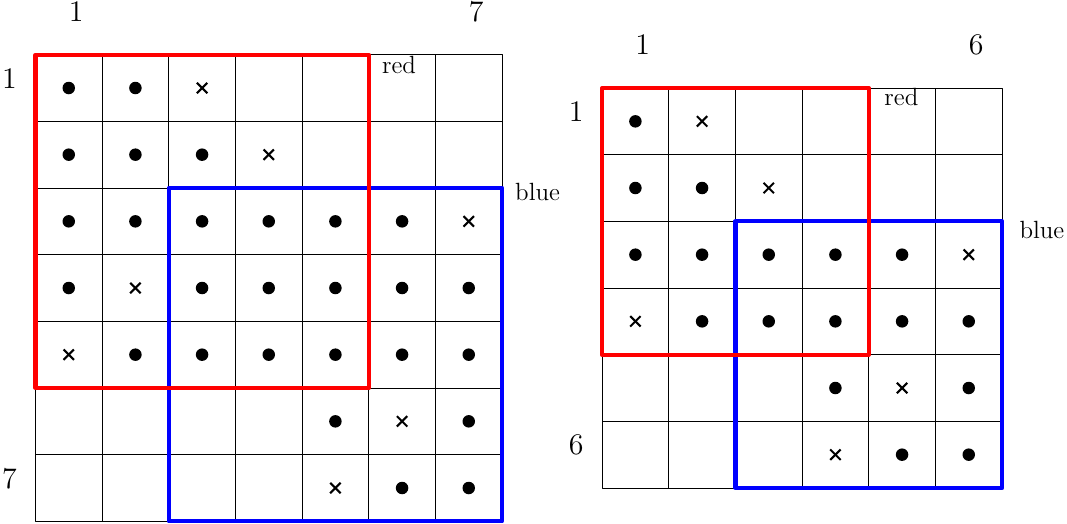}
    \caption{On the left, $\gr{e}{w}$, where $w=3472165$. On the right, $\gr{e}{u}$, where $u=236154$. The permutation $u$ is obtained from $w$ by deleting 2 from the one-line notation of $w$ and shifting the other numbers appropriately. As \cref{prop:deleteinternaldots}} asserts, the spanning corners and bounding boxes of $\gr{e}{w}$ are in bijection with the spanning corners and bounding boxes of $\gr{e}{u}$, respectively, and this bijection is color-preserving.
    \label{fig:deleteInternalDot}
\end{figure}

\begin{prop} \label{prop:deleteinternaldots} Suppose $w \in S_n$ avoids 3412 and let $\wire{i}{w_i}$ be a non-corner of $\gr{e}{w}$. Let $u \in S_{n-1}$ be the permutation $\delta_i(j) \mapsto \delta_{w_i}(w_j)$. Let $\overline{S}_w$ and $\overline{S}_u$ denote the spanning corners of $\gr{e}{w}$ and $\gr{e}{u}$, respectively.

\begin{enumerate}
    \item The pair $\wire{j}{w_j} \in \overline{S}_w$ if and only if $\wire{\delta_i(j)}{\delta_{w_i}(w_j)} \in \overline{S}_u$.
    \item The map $\beta:\{B_{j, w_j}: (j, w_j) \in \overline{S}_w\} \to \{B_{k, u_k}: (k, u_k) \in \overline{S}_u\}$ defined by $B_{j, w_j} \mapsto B_{\delta_i(j),\delta_{w_i}(w_j)}$ is a color-preserving bijection.
    \item Order the bounding boxes of $\gr{e}{w}$ according to the row of their northwest corner. Then $\beta$ also preserves this ordering.
\end{enumerate}

\end{prop}

\begin{proof}
Let $\alpha_a$ denote the inverse of $\delta_a$, so that $w$ sends $\alpha_i(j)$ to $\alpha_{u_i}(w_j)$ (since $(i, w_i)$ is a non-corner, $i, w_i \neq n$ and this is well-defined). Note that $\delta_a$ and $\alpha_a$ are order-preserving. Let $\delta:=\delta_i \times \delta_{w_i}$ and $\alpha:=\alpha_i \times \alpha_{w_i}$.

We first show that $\delta$ is a bijection from the non-corner pairs $(j, w_j)$ of $\gr{e}{w}$ with $j \neq i$ to the non-corner pairs $(k, u_k)$ of $\gr{e}{u}$. Recall that $(j, w_j)$ is a non-corner of $\gr{e}{w}$ if and only if $(j, w_j)$ is sandwiched by an inversion of $w$. Moreover, every non-corner pair $(j, w_j)$ is sandwiched by an inversion $\lrangle{k, l}$ where $(k, w_k)$ and $(l, w_l)$ are corners of $\gr{e}{w}$. Indeed, choose the smallest $k$ such that $\lrangle{k, j}$ is an inversion and the largest $l$ such that $\lrangle{j, l}$ is an inversion. Then $(k, w_k)$ and $(l, w_l)$ are both corners.

Let $j \neq i$ and suppose $(j, w_j)$ is sandwiched by an inversion $\lrangle{k, l}$ of $w$. We can choose $k, l$ so that $(k, w_k)$ and $(l, w_l)$ are corners of $\gr{e}{w}$; in particular, neither is equal to $i$.  Because $\delta_a$ is order-preserving, $\lrangle{\delta_i(k), \delta_i(l)}$ is an inversion of $u$ and sandwiches $\delta(j, w_j)$. Similarly, if $(j, u_j)$ is sandwiched by an inversion $\lrangle{k, l}$ of $u$, then $\alpha(j, u_j)$ is sandwiched by the inversion $\lrangle{\alpha_i(k), \alpha_i(l)}$ of $w$. As $\delta$ and $\alpha$ are inverses, we are done.

This also implies that $\delta$ is a bijection from the corner pairs $(j, w_j)$ of $\gr{e}{w}$ to the corner pairs $(k, u_k)$ of $\gr{e}{u}$.

Next, we show that $\delta$ respects containment of spans for corner pairs of $w$. 

Suppose $\spn{k, w_k}$ is contained in $\spn{j, w_j}$ and both pairs are corners of $\gr{e}{w}$. We may assume that $k \neq j$ (otherwise, $\delta$ clearly respects containment of spans) and that $j<w_j$ (otherwise, we can consider instead $w^{-1}$, which avoids 3412 also, and $u^{-1}$). By assumption, $k, w_k \in [j, w_j]$, so $\delta_i(k)$ and $\delta_{w_i}(w_k)$ are in $[j-1, w_j]$. We have 

\[\spn{\delta_i(j), \delta_{w_i}(w_j)}= \begin{cases}
[j, w_j] &   \specialcell{ \text{ if }j<i, w_j<w_i \hspace{1in} \text{(Case I)}}\\
[j, w_j-1] & \specialcell{\text{ if }j<i, w_j>w_i  \hfill \text{(Case II)}}\\
[j-1, w_j] & \specialcell{\text{ if }j>i, w_j<w_i  \hfill \text{(Case III)}}\\
[j-1, w_j-1] & \specialcell{\text{ if }j>i, w_j>w_i.  \hfill \text{(Case IV)}}\\
\end{cases}
\]

Case I: The only way $\spn{\delta(j, w_j)}$ could fail to contain $\spn{\delta(k, w_k)}$ here is if $j-1 \in \{\delta_i(k),\delta_{w_i}(w_k)\}$. If $j-1=\delta_i(k)$, then $k=j$, a contradiction. If $j-1=\delta_{w_i}(w_k)$, then $w_k=j$ and $w_k>w_i$. But $w_k<w_j$ and $w_j<w_i$ by assumption, so we reach a contradiction.

Case II: The only way $\spn{\delta(j, w_j)}$ could fail to contain $\spn{\delta(k, w_k)}$ here is if $j-1$ or $w_j$ is in $ \{\delta_i(k),\delta_{w_i}(w_k)\}$. We will show that $w_j$ is not in $\{\delta_i(k),\delta_{w_i}(w_k)\}$ by contradiction; the other case is similar. 

Suppose $w_j \in \{\delta_i(k),\delta_{w_i}(w_k)\}$. Since $w_k < w_j$, we must have $w_j=\delta_i(k)$, which means $k=w_j$ and $k<i$. Notice that $\lrangle{j, k}$ is an inversion of $w$, as is $\lrangle{j, i}$. Since $(k, w_k)$ is not sandwiched by any inversions of $w$, we must have $w_k< w_i$. To summarize, $j<k<i$ and $w_k<w_i<w_j$. This means $w_j w_k w_i$ form a 312 pattern in $w$.

Note that $w$ cannot have any inversions $\lrangle{a, j}$ or $\lrangle{k, a}$, since this would result in $(j, w_j)$ or $(k, w_k)$, respectively, being sandwiched by an inversion. We further claim that if $a$ forms an inversion with $k$ and $i$, then it must also form an inversion with $j$. Indeed, if $a$ forms an inversion with $k$ and $i$, then $a<k<i$ and $w_a>w_i>w_k$. If $j<a$ and $w_j<w_a$, then $w_j w_a w_k w_i$ form a 3412 pattern; similarly if $j>a$ and $w_j>w_a$. 

Consider $a<j$. From above, we know that $\lrangle{a, j}$ is not an inversion. If $w_a>w_i$, then $\lrangle{a, i}$ and $\lrangle{a, k}$ are both inversions. This combination is impossible, so $w[j-1] \subseteq [w_i-1]$. Also, any $a \in [j+1, k-1]$ with $w_a>k$ forms an inversion with $k$ and $i$ but not $j$, which is impossible. So $w[j+1, k-1] \subseteq[k]$. Since $w_j=k$ and $w_k< w_i<k$, we conclude $w[k] = [k]$. But $i>k$ and $w_i<k$, a contradiction.

Case III: The span of $\delta(k, w_k)$ is contained in $[j-1, w_j]$ by assumption.

Case IV: The only way $\spn{\delta(j, w_j)}$ could fail to contain $\spn{\delta(k, w_k)}$ here is if $w_j \in \{\delta_i(k),\delta_{w_i}(w_k)\}$. The argument that this cannot happen is similar to Case I; we leave it to the reader.

Finally, we will show that $\alpha$ respects span containment for corner pairs of $u$. This completes the proof of \emph{(1)}: suppose $(j, w_j)$ is a spanning corner of $\gr{e}{w}$ and $\spn{\delta(j, w_j)} \subseteq \spn{(a, u_a)}$ for a spanning corner $(a, u_a)$. Note that $\delta(j, w_j)$ is a corner. Since $\alpha$ respects span containment for corners, $\spn{j, w_j} \subseteq \spn{\alpha(a, u_a)}$. By maximality of $\spn{j, w_j}$, we have $\spn{j, w_j}=\spn{\alpha(a, u_a)}$. In particular, $\spn{\alpha(a, u_a)} \subseteq \spn{j, w_j}$, so since $\delta$ preserves span containment for corners, the span of $(a, u_a)$ is contained in the span of $\delta(j, w_j)$. So $\delta(j, w_j)$ is a spanning corner of $\gr{e}{u}$. Reversing the roles of $w$ and $u$ in the above argument shows that $\alpha(j, u_j)$ is a spanning corner of $\gr{e}{w}$ if $(j, u_j)$ is a spanning corner of $\gr{e}{u}$. So $\delta$ is a bijection between the spanning corners of $\gr{e}{w}$ and the spanning corners of $\gr{e}{u}$.

 Suppose $\spn{k, u_k}$ is contained in $\spn{j, u_j}$ and both pairs are corners of $\gr{e}{u}$. Again, we may assume that $k \neq j$ (otherwise, $\delta$ clearly respects containment of spans) and that $j<u_j$ (otherwise, consider instead $w^{-1}$, which avoids 3412, and $u^{-1}$). By assumption, $k, u_k \in [j, u_j]$, so $\alpha_i(k)$ and $\alpha_{w_i}(u_k)$ are in $[j, u_j+1]$. We have 

\[\spn{\alpha_i(j), \alpha_{w_i}(u_j)}= \begin{cases}
[j, u_j] &   \specialcell{ \text{ if }j<i, u_j<w_i \hspace{1in} \text{(Case I')}}\\
[j, u_j+1] & \specialcell{\text{ if }j<i, u_j \geq w_i  \hfill \text{(Case II')}}\\
[j+1, u_j] & \specialcell{\text{ if }j \geq i, u_j<w_i  \hfill \text{(Case III')}}\\
[j+1, u_j+1] & \specialcell{\text{ if }j \geq i, u_j \geq w_i.  \hfill \text{(Case IV')}}\\
\end{cases}
\]

Case I': The only way $\spn{\alpha(k, u_k)}$ could fail to be contained in $\spn{\alpha(j, u_j)}$ is if $u_j+1 \in \{\alpha_i(k), \alpha_{w_i}(u_k)\}$. Suppose that this occurs. Since $u_k<u_j$, we must have $k=u_j$ and $k \geq i$. Also, $u_k<u_j<w_i$. So $\alpha(k, u_k)=(k+1, u_k)$ and $\alpha(j, u_j)=(j, u_j)$. To summarize, we have $j<i<k+1$ and $w_{k+1}<w_j<w_i$. So $w_j w_i w_{k+1}$ form a 231 pattern. 

Because $(k+1, w_k)$ and $(j, w_j)$ are corners and are not sandwiched by any inversions, $w$ has no inversions of the form $\lrangle{a, j}$ or $\lrangle{k+1, a}$. Also, any $a$ forming an inversion with $i$ and $k+1$ but not $j$ would give rise to a 3412 pattern, so no such $a$ exist. 

Consider $a >k+1$. If $w_a < w_j$, then $a$ would either form an inversion with $k+1$, which is impossible, or $a$ would form an inversion with both $i$ and $j$ but not $k+1$, which is also impossible. So $w[k+2, n] \subseteq [w_j, n]$. Since $j<k+2$ and $w_j=k$, in fact $w[k+2, n] \subseteq [k+1, n]$. Notice that $w_i \geq k+1$ and $i < k+1$, so we can refine this further to $w[k+2, n] = [k+1, n] \setminus \{w_i\}$. But $(i, w_i)$ is sandwiched be some inversion $\lrangle{a, b}$, so $a<i<k+1$ and $w_a >w_i \geq k+1$. This is clearly a contradiction.

Case II': By assumption, $\spn{\alpha(k, u_k)} \subseteq [j, u_j+1]$, so the claim is true.

Case III': The only way $\spn{\alpha(k, u_k)}$ could fail to be contained in $\spn{\alpha(j, u_j)}$ is if $j$ or $u_j+1$ were in $\{\alpha_i(k), \alpha_{w_i}(u_k)\}$. Suppose that $j \in \{\alpha_i(k), \alpha_{w_i}(u_k)\}$. Since $k>j$, this means $u_k=j=\alpha_{w_i}(u_k)$ and $w_i>u_k$. So $\alpha(k, u_k)=(k+1, u_k)$ and we have $i<j+1<k+1$ and $w_{k+1}<w_{j+1}<w_i$. This means that $(j+1, w_{j+1})$ is sandwiched by the inversion $\lrangle{i, k+1}$, a contradiction. The other case is similar.

Case IV': The only way $\spn{\alpha(k, u_k)}$ could fail to be contained in $\spn{\alpha(j, u_j)}$ is if $j \in \{\alpha_i(k), \alpha_{w_i}(u_k)\}$. This is similar to Case I', so we leave it to the reader.

For \emph{(2)}: The map 

\begin{align*}
    \beta: \{B_{j, w_j}: (j, w_j) \in \overline{S}_w\} &\to \{B_{k, u_k}: k=1, \dots, n\}\\
    B_{j, w_j} & \mapsto B_{\delta_i(j),\delta_{w_i}(w_j)}
\end{align*}

 is well-defined and injective because $\alpha$ and $\delta$ preserve span containment (for corners) and thus also preserve equality of spans (for corners). So $B_{\delta_i(j),\delta_{w_i}(w_j)}=B_{\delta_i(k),\delta_{w_i}(w_k)}$ if and only if $B_{j, w_j}=B_{k, w_k}$. Its image is the bounding boxes of $\gr{e}{u}$, since $\delta$ is a bijection between spanning corners of $\gr{e}{w}$ and spanning corners of $\gr{e}{u}$.
 
 We will show $\beta$ preserves the colors of the boxes by contradiction. Suppose the color of $\beta(B_{j, w_j})$ differs from the color of $B_{j, w_j}$. This situation means that the relative order of $j, w_j$ must be different from that of $\delta_i(j), \delta_{w_i}(w_j)$. This can only happen if $\min(j, w_j)$ is not shifted down by $\delta_a$ (for the appropriate $a \in \{i, w_i\}$), $\max(j, w_j)$ is shifted down by $\delta_b$ (for $b \in \{i, w_i\} \setminus \{a\}$) and $|j-w_j|\leq 1$. That is, $\min(j, w_j)< a$ and $\max(j, w_j)>b$. If $j=w_j$, this implies $(j, w_j)$ is spanned by $(i, w_i)$, a contradiction. Otherwise, this implies $(i, w_i)$ is spanned by $(j, w_j)$. Because $|j-w_j|=1$, the only possibility for this is that $i=w_j$ and $w_i=j$, so the spans are equal. But $(j, w_j)$ is a corner and $(i, w_i)$ is not, a contradiction.

This means $\beta$ sends green bounding boxes to green bounding boxes, blue to blue, and red to red. It also sends purple to purple: suppose $(j, w_j)$ is a spanning corner and $B_{j, w_j}$ is purple. Then $(w_j, j)$ is also a spanning corner of $\gr{e}{w}$ and is not equal to $(j, w_j)$. Since the span of $(j, w_j)$ and $(w_j, j)$ are the same, the span of $\delta(j, w_j)$ and $\delta(w_j, j)$ are the same; since $\delta$ is a bijection on spanning corners, $\delta(j, w_j) \neq \delta(w_j, j)$. So $B_{\delta(j, w_j)}=B_{\delta(w_j, j)}$ and this bounding box is both red and blue.

For \emph{(3)}: Suppose $B_{j, w_j}$ and $B_{k, w_k}$ are two bounding boxes of $\gr{e}{w}$ and $B_{j, w_j}$ precedes $B_{k, w_k}$ in the order given. That is, $\min(j, w_j)<\min(k, w_k)$. Suppose for the sake of contradiction that $\min(\delta_i(j), \delta_{w_i}(j)) \geq \min(\delta_i(k), \delta_{w_i}(k))$. In fact, because $\delta_a$ shifts numbers by at most 1, the only possibility is that $\min(\delta_i(j), \delta_{w_i}(j)) = \min(\delta_i(k), \delta_{w_i}(k))$. Since $\delta(j, w_j)$ and $\delta(k, w_k)$ are both spanning corners of $\gr{e}{u}$ and thus have maximal span, this implies that the span of $\delta(j, w_j)$ is equal to the span of $\delta(k, w_k)$. So $B_{\delta(j, w_j)}=B_{\delta(k, w_k)}$. Since $\beta$ is a bijection, this implies $B_{j, w_j}=B_{k, w_k}$, a contradiction. 


\end{proof}

\cref{prop:alternatingboxesDiag} follows as a corollary.


\begin{proof}[Proof of \cref{prop:alternatingboxesDiag}]
Note that no bounding boxes of $\gr{e}{w}$ are green, since this would imply $w$ is contained in some maximal parabolic subgroup.

Repeatedly apply the operation of \cref{prop:deleteinternaldots} to $w$ until you arrive at a permutation $u$ with no non-corner pairs.

The permutation $u$ will avoid 3412. Indeed, one-line notation for $u$ can be obtained from $w$ by repeatedly deleting some number $a$ and applying $\delta_a$ to the remaining numbers. Since $\delta_a$ preserves order, any occurrence of 3412 in $u$ would imply an occurrence of 3412 in $w$. It will also avoid 321, since if $u_i u_j u_k$ form a 321 pattern, $(j, w_j)$ is sandwiched by the inversion $\lrangle{i, k}$ and thus is a non-corner pair.

By \cref{prop:alternatingboxes321}, no bounding boxes of $\gr{e}{u}$ are purple and they alternate between red and blue (when ordered by the row of the northwest corner). \cref{prop:deleteinternaldots} implies that the bounding boxes of $\gr{e}{w}$ are in bijection with the bounding boxes of $\gr{e}{u}$ and that this bijection preserves the coloring and ordering of the bounding boxes. So no bounding boxes of $\gr{e}{w}$ are purple, and they alternate between red and blue.
\end{proof}

We now are ready to prove \cref{prop:alternatingBoxesAnti}.

\begin{proof}[Proof of \cref{prop:alternatingBoxesAnti}]
Since $v$ avoids 2143, $w_0 v$ avoids 3412. By assumption, $w_0v$ is not contained in a maximal parabolic subgroup, so by \cref{prop:alternatingboxesDiag}, the bounding boxes of $\gr{e}{w_0v}$ alternate in color between red and blue (and none are purple) when ordered by the row of their northeast corner. By \cref{rmk:bounding-boxes}, reversing the columns of these bounding boxes gives the bounding boxes of $\gr{v}{w_0}$, now ordered according to the row of their northwest color. Since the bounding boxes of $\gr{v}{w_0}$ have the same color as the corresponding bounding boxes of $\gr{e}{w_0v}$, the proposition follows.
\end{proof}

\section{Proof of Proposition~\ref{prop:deleteDotDetAnti}}\label{sec:proof2}

We apply a similar technique as in the above section.  Rather than proving \cref{prop:deleteDotDetAnti} directly, we instead prove the following:

\begin{prop} \label{prop:deleteDotDet}
Let $w\in S_n$ be 4231- and 3412-avoiding, and choose $i\in [n]$. Let $u \in S_{n-1}$ be the permutation obtained from $w$ by deleting $w_i$ from $w$ in one-line notation and shifting appropriately (that is, $u: \delta_i(j) \mapsto \delta_{w_i}(w_j)$). If $(i, v_i)$ is not a spanning corner of $\gr{e}{w}$, then $\gr{e}{u}=\gr{e}{w}_i^{w_i}$. Further, for all $i$, \begin{equation} \label{eqn:deleteDotDet}
\det(\res{M}{\gr{e}{u}})=\det(\res{M}{\gr{e}{w}_{i}^{w_i}}).\end{equation}
\end{prop}

We prove this using a sequence of lemmas.


\begin{lem} \label{lem:boxIntersectionsMoreGraphs}
Let $w \in S_n$ be 3412-avoiding. Let $\wire{i}{w_i}$ be a spanning corner, and let $q=\min(i, w_i)$. Let $N:=\{j \in [n]: j, i \text{ form an inversion of }w\} \cup \{i\}$. Say $ N=\{k_1, \dots ,k_m\}$ and let $\rho: w(N) \to [m]$ be the unique order-preserving bijection between the two sets. Let $u$ be the permutation of $[m]$ whose one line notation is $\rho(w_{k_1}) \rho(w_{k_2}) \cdots \rho(w_{k_m})$.  Then $B_{i, w_i} \cap \gr{e}{w}=\{(r, c): (r, c)- (q-1,q-1) \in \gr{e}{u}\}$.
\end{lem}

\begin{ex} \label{ex:boxIntersection}
Let $w=3472165$ (see \cref{fig:deleteInternalDot} for a picture of $\gr{e}{w}$). Choose the spanning corner $(3, 7)$. Then $N=\{3, 4, 5, 6, 7 \}$ and $u=52143$. The graph $\gr{e}{u}$ is pictured below.

\begin{center}
\includegraphics[height=0.15\textheight]{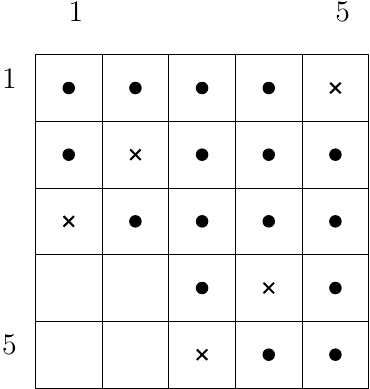}
\end{center}

The part of $\gr{e}{w}$ that lies in $B_{3, 7}$ is identical to $\gr{e}{u}$ (up to translation along the diagonal).
\end{ex}

\begin{proof} [Proof of \cref{lem:boxIntersectionsMoreGraphs}]
We may assume $i<w_i$, so $B_{i, w_i}$ is blue; otherwise, we can consider $w^{-1}$, which will still avoid $3412$, $u^{-1}$, which can be obtained from $w^{-1}$ by the same procedure as $u$ is obtained from $w$, and the bounding box $B_{w_i, i}$, which is blue. The intersection $\gr{e}{w^{-1}} \cap B_{w_i, i}$ is simply the transpose of $B_{i, w_i} \cap \gr{e}{w}$.

We may also assume that $w$ is not contained in a maximal parabolic subgroup of $S_n$. (If it were, we could consider just the block of $\gr{e}{w}$ containing $(i, w_i)$ and argue just about that block.) We may further assume that $\gr{e}{w}$ has more than one bounding box. By \cref{prop:alternatingboxesDiag}, the bounding boxes of $\gr{e}{w}$ alternate between blue and red, and none are green or purple.

Notice that $ N$ is contained in $[i, n]$, since if $\lrangle{j, i}$ were an inversion of $w$, $(j, w_j)$ would span $(i, w_i)$. So $k_1=i$, and $u_1=\rho(w_i)$. Because $\rho$ is order preserving, $\lrangle{1, k}$ is an inversion of $u$ for all $k \in [m]$, which implies $\rho(w_i)=m$.  

We have $j \in  N$ precisely when $(j, w_j)$ lies southwest of $(i, w_i)$ in the plane, since there are no $(j, w_j)$ to the northeast. To obtain $\gro{u}$ from $\gro{w}$, delete all rows and columns of the $n \times n$ grid which have a cross to the north or east of $B_{i, w_i}$ (that is, a cross $(j, w_j)$ with $j<i$ or $w_j>w_i$) and renumber remaining rows and columns with $[m]$. Note that $|i-w_i|=|1 - m|$ because for every row above $i$ that is deleted, a column to the left of $w_i$ is deleted. So $B_{i, w_i}$ is an $m \times m$ square, which we can identify with the $m \times m$ square containing $\gr{e}{u}$ by relabeling rows and columns. Also, these deletions take the corners (resp. non-corners) of $\gr{e}{w}$ with $j \in  N$ to corners (resp. non-corners) of $\gr{e}{u}$.

Thus, it suffices to check the following: if $(r, c) \in B_{i, w_i}$ is sandwiched by an inversion $\lrangle{i, j}$, where $(j, w_j)$ is a corner of $\gr{e}{w}$, then the corresponding square of $\gr{e}{u}$ is sandwiched by an inversion of $u$.

First, let $B_{a, w_a}$ and $B_{b, w_b}$ be the red bounding boxes immediately preceding and following $B_{i, w_i}$, respectively, in the usual order on bounding boxes. If $\lrangle{i, j}$ is an inversion of $w$, then $(j, w_j) \in B_{a, w_a} \cup B_{i, w_i} \cup B_{b, w_b}$. Indeed, suppose $(j, w_j)$ is a corner such that $\lrangle{i, j}$ is an inversion, and $(j, w_j) \notin B_{a, w_a} \cup B_{i, w_i} \cup B_{b, w_b}$. Then either $w_j<w_a$ or $j>b$; otherwise $(j, w_j)$ would not be in the union of bounding boxes for $w$, a contradiction of \cref{lem:boundingboxesDiag}. If $j>b$, then there is a blue bounding box $B_{d, w_d}$ immediately following $B_{b, w_b}$ in the usual order of bounding boxes. One can check that $i<d<b<j$ and $w_i w_d w_b w_j$ forms a 3412 pattern. If $w_j<w_a$, there is a blue bounding box $B_{d, w_d}$ immediately preceding $B_{a, w_a}$, and one can check that $d<i<j<a$ and $w_d w_i w_j w_a$ forms a 3412 pattern. If $(j, w_j)$ is not a corner but $\lrangle{i, j}$ is an inversion, then $(j, w_j)$ is sandwiched by an inversion $\lrangle{i, k}$ where $k$ is a corner, so $(j, w_j)$ is also in the union of the three bounding boxes.

This implies that if $\lrangle{i, j}$ is an inversion of $w$ such that $(j, w_j)$ is a corner, then either $(j, w_j) \in B_{i, w_i}$ or $j \in \{a, b\}$. So either $w_i \geq j$ or $i \leq w_j$.

Suppose $w_i \geq j$ (so $(j, w_j)$ is either in $B_{i, w_i}$ or $j=a$). We claim no rows between $i$ and $j$ are deleted. Indeed, a row between $i$ and $j$ is deleted only if there is a dot $(k, w_k)$ to the east of $B_{i, w_i}$ with $i<k<j$. Necessarily, $w_k>w_i$. If there is a dot to the east of $B_{i, w_i}$, then $B_{i, w_i}$ is not the last bounding box. By \cref{prop:alternatingboxesDiag}, the following bounding box $B_{s, w_s}$ is red. One can check that $w_i w_k  w_j  w_s$ is a 3412 pattern, a contradiction. 
By a similar argument, if $(j, w_j)$ is a corner such that $\lrangle{i, j}$ is an inversion of $w$, and $i \leq w_j$ (so $(j, w_j)$ is either in $B_{i, w_i}$ or $j=b$), then no column between $w_j$ and $w_i$ is deleted.

Now, if a corner $(j, w_j)$ is in $B_{i, w_i}$, then $i \leq w_j$ and $w_i \geq j$, so the relative position of $(i, w_i)$ and $(j, w_j)$ is the same as the relative position of the images of $(i, w_i)$ and $(j, w_j)$ after deletion. So if $(a, b) \in B_{i, w_i}$ is sandwiched by $\lrangle{i, j}$, the corresponding square in $\gr{e}{u}$ is sandwiched by the image of $\lrangle{i, j}$.

If $j=a$ (resp. $j=b$), then $\rho(w_j)=1$ (resp. $j=k_m$). That is, the image of $(j, w_j)$ after deletion is $(j, 1)$ (resp. $(m, w_j)$). This is because $(j, w_j)$ is the west-most (resp. south-most) cross forming an inversion with $(i, w_i)$. So if $(a, b) \in B_{i, w_i}$ is sandwiched by $\lrangle{i, j}$, the corresponding square in $\gr{e}{u}$ is sandwiched by the image of $\lrangle{i, j}$.

\end{proof}

From \cref{lem:boxIntersectionsMoreGraphs}, we can derive the following.

\begin{lem} \label{lem:invNotInBoxOK}
Let $w \in S_n$ be 3412-avoiding, and let $\wire{i}{w_i}$ be a spanning corner. Suppose $(r, c) \notin B_{i, w_i}$ is sandwiched by an inversion involving $i$. Then $(r, c)$ is also sandwiched by some inversion $\lrangle{a, b}$ where neither $(a, w_a)$ nor $(b, w_b)$ are in $B_{i, w_i}$.
\end{lem}

\begin{proof}
We will assume that $B_{i, w_i}$ is blue; otherwise, we consider $w^{-1}$ instead. We also assume that $w \in S_n$ is not contained in any parabolic subgroup; if it were, $\gr{e}{w}$ is block-diagonal and we can argue for each block individually. The lemma is vacuously true if $\gr{e}{w}$ has a single box, so we may assume it does not. By \cref{prop:alternatingboxesDiag}, the bounding boxes of $\gr{e}{w}$ alternate between red and blue, and none are purple.

Suppose $B_{a, w_a}$ and $B_{b, w_b}$ are the (red) bounding boxes immediately preceding and following $B_{i, w_i}$, respectively. As in \cref{lem:boxIntersectionsMoreGraphs}, if $j$ forms an inversion with $i$, then $(j, w_j) \in B_{a, w_a} \cup B_{i, w_i} \cup B_{b, w_b}$.

This implies that the positions $(r, c)$ satisfying the conditions of the lemma are contained $B_{a, w_a} \cup B_{b, w_b}$. The positions $(r, c) \subseteq B_{a, w_a}$ satisfying the conditions of the lemma are exactly those with $w_a\leq c<i$ and $i \leq r \leq a$. By \cref{lem:boxIntersectionsMoreGraphs}, $\gr{e}{w} \cap B_{a, w_a}$ is the graph of another interval $[e, u]$ where $u \in S_{m}$; since $B_{a, w_a}$ is red, $u$ sends $m$ to 1. This means that $u$ is greater than the permutation $2 3 \cdots m 1$. In particular, this means that positions $(q, q+1)$ for $q=1, \dots, m-1$ are in $\gr{e}{u}$, so the upper off-diagonal of $B_{a, w_a}$ is in $\gr{e}{w}$. Thus, $(i-1, i)$ is sandwiched by some inversion of $w$, implying there is a dot $(j, w_j)$ northeast of $(i-1, i)$. This dot is necessarily not in $B_{i, w_i}$, and the inversion $(j, a)$ sandwiches all of the positions $(r, c) \subseteq B_{a, w_a}$ satisfying the conditions of the lemma.

The argument for the positions $(r, c) \subseteq B_{b, w_b}$ satisfying the conditions of the lemma is essentially the same.
\end{proof}

We can now prove \cref{prop:deleteDotDet}.

\begin{proof}[Proof of \cref{prop:deleteDotDet}] Let $D=\gr{e}{w}_{i}^{w_i}$.

 Consider $\gr{e}{w}$ drawn in an $n\times n$ grid with the positions $(j, w_j)$ marked with crosses and all others marked with dots. Recall that $D$ is the collection of crosses and dots obtained from $\gr{e}{w}$ by deleting row $i$ and column $w_i$ and renumbering rows by $\delta_i$ and columns by $\delta_{w_i}$. Note that the crosses of $D$ are in positions $(j, u_j)$.

 If $(i, w_i)$ was an internal dot of $\gr{e}{w}$, then all dots in $D$ are sandwiched by an inversion of $u$, so $D=\gr{e}{u}$. In this case, $\res{M}{D}=\res{M}{\gr{e}{u}}$, so the determinants are equal.
 
 If $(i, w_i)$ is a corner but not a spanning corner of $\gr{e}{w}$, we claim we again have $D= \gr{e}{u}$. Suppose $(i, w_i)$ is contained in a blue bounding box $B_{k, w_k}$ (if it is only contained in a red bounding box, we can consider $w^{-1}$ and the transpose of $M$ instead). Notice that since $(i, w_i)$ is a corner, we only have inversions $\lrangle{r, i}$ with $r<i$. Further, there are no inversions $\lrangle{r, i}$ where $r<k$ or $w_r>w_k$; otherwise, we can find an occurrence of 3412. For example, if there were an inversion $\lrangle{r, i}$ with $r<k$, then there must be a red bounding box $B_{a, w_a}$ immediately preceding $B_{k, w_k}$ which overlaps with it. One can check that $r<k<a<i$ and $w_r w_k w_a w_i$ is an occurrence of 3412. A similar argument works for the $w_r>w_k$ case. Thus, if $\lrangle{r, i}$ is an inversion, $\lrangle{k, r}$ is an inversion as well, so every dot sandwiched by an inversion $(r, i)$ is also sandwiched by the inversion $(k, i)$.
 
 In particular, there are no crosses above or to the right of the rectangle with corners $(k, w_k)$ and $(i, w_i)$, since any such cross would be inversions $\lrangle{r, i}$ where $\lrangle{k, r}$ is not an inversion. There are also no crosses northeast of $(k, w_k)$, since such a cross would span $(k, w_k)$. This means there are no dots above or to the right of the rectangle with corners $(k, w_k)$ and $(i, w_i)$. So in fact it suffices to show that all dots in the rectangle with corners $(k, w_k)$ and $(i-1, w_i +1)$ are sandwiched by an inversion of $v$ that does not involve $i$; that inversion will correspond to an inversion in $u$.  If $w_{i-1}<w_i$ or if $w^{-1}(w_i+1)>i$, this is true. Otherwise, we have that $(i-1, w_{i-1})$ and $(w^{-1}(w_i+1), w_i+1)$ both lie in the rectangle with corners $(k, w_k)$ and $(i-1, w_i +1)$. If these points are distinct, then $w_{i-1}>w_i+1$, so $w_k~ w_{i}+1 ~w_{i-1} ~ w_i$ form a 4231 pattern, which is impossible. So we must have $w_{i-1}w_i +1$, which means all points in the rectangle with corners $(k, w_k),(i-1, w_i +1)$ are sandwiched by the inversion $\lrangle{k, i-1}$.
 
Now, suppose $\wire{i}{w_i}$ is a spanning corner and $B_{i, w_i}$ is blue (if it is red, we can consider $w^{-1}$ and the transpose of $M$ instead). If $w_{i+1}=w_i-1$, we again have $A=B$. If not, then the $D$ is not equal to $\gr{e}{u}$;  $D=\gr{e}{u} \sqcup \{(r, c): (r, c) \text{ sandwiched only by inversions involving } i\}$. We will show that if $(r, c)$ is sandwiched only by inversions involving $i$, then $m_{r, c}$ will not appear in $\det(\res{M}{D})$. This will imply that $\det(\res{M}{D})$ agrees with $\det(\res{M}{\gr{e}{u}})$ for all $(n-1) \times (n-1)$ matrices $M$.

Let $I:= \{(r, c): (r, c) \text{ sandwiched only by inversions involving } i\}$. By \cref{lem:invNotInBoxOK}, $I \subseteq B_{i, w_i}$. We would like to show that $D$ is block-upper triangular, and that all $(r,c)\in I$ are in blocks that are not on the main diagonal.  This would imply that $\det(\res{M}{D})$, as claimed, does not contain $m_{r, c}$, as it is the product of the determinants of the blocks on the diagonal.  To verify this, we just need to show that $D \cap B_{i, w_i}$ is block-upper triangular and for all $(r,c)\in I\cap B_{i,w_i}$, $(r,c)$ is in a block that is not on the main diagonal. By \cref{lem:boxIntersectionsMoreGraphs}, $\gr{e}{w} \cap B_{i, w_i}$ is another $\gr{e}{x}$. So $D \cap B_{i, w_i}$ is simply $\gr{e}{x}$ with the first row and last column removed. Thus, it suffices to prove the following lemma.

\begin{lem}
Let $x \in S_m$ be $4231$ avoiding. Suppose $x_1=m$ and $x_2 \neq m-1$. Let $Q$ be the region obtained by removing the first row and last column of $\gr{e}{x}$. Then $Q$ is block-upper triangular, and for all positions $(r, c)$ sandwiched only by inversions involving $1$, $(r,c)$ is in a block not on the main diagonal. 
\end{lem}

\begin{proof}
First, recall that $\gr{e}{x}$ contains the lower off-diagonal $\{(j, j-1): j=2, \dots, m\}$, since $x > m 1 2 3 \cdots (m-1)$.

Call a dot $(i, x_i)$ \emph{leading} if it is not southwest of any $(j, x_j)$ besides $(1, m)$. (For example, in \cref{ex:boxIntersection}, the leading dots of $\gr{e}{u}$ are $(2, 2)$ and $(4,4)$.) We claim that for each leading dot $(i, x_i)$, position $(x_i+2, x_i)$ is not sandwiched by an inversion of $x$.

We show this first for the northmost leading dot $(2, w_2)$. Notice that there cannot be any $(j, x_j)$ with $2<j\leq x_2+1$ and $x_j>x_2$; in this case, avoiding $4231$ would imply that $(j, x_2)$ is not sandwiched by any inversion, which would in turn imply that $(x_2+1, x_2)$ is not sandwiched by any inversion. But the lower off-diagonal is contained in $\gr{e}{x}$, so this is a contradiction. This means that for all $2<j\leq x_2+1$, $(j, x_j)$ is contained in the square with opposing corners $(2, 1)$ and $(x_2+1, x_2)$. In particular, there is a dot $(j, x_j)$ in every row and every column of this square. So for $j>x_2+1$, we have $x_j>x_2$, which implies $(x_2+2, x_2)$ is not sandwiched by an inversion of $x$. In other words, there are no elements of $\gr{e}{x}$ in columns $1, \dots, x_2$ and rows $x_2+2, \dots, m$. This implies that $Q$ is block-upper triangular, and the first diagonal block has northwest corner $(2, w_2)$.

Note that the next leading dot is in row $i_2:=x_2+2$. We can repeat the above argument with this dot to reach the analogous conclusion: the second block of $Q$ has northwest corner $(i_2, x_{i_2})$. We can continue with the remaining leading dots to see that $Q$ is block-upper triangular with northwest corners of each diagonal block given by the leading dots.

Notice that the union of the diagonal blocks contains every position $(r, c)$ sandwiched by an inversion not involving $1$. Thus, the complement of the union of diagonal blocks in $Q$ is exactly the positions $(r, c)$ which are sandwiched only by an inversion involving $1$. This finishes the proof of the lemma.
\end{proof}

\end{proof}

Finally, we can prove \cref{prop:deleteDotDetAnti}.

\begin{proof}[Proof of \cref{prop:deleteDotDetAnti}]
Let $w:=w_0v$ and $u:=w_0x$. Notice that $u: \delta_i(j) \mapsto \delta_{w_i}(w_j)$ and that $w$ avoids 3412 and 4231.

It follows from \cref{lem:reverseCols} that restricting $M$ to $\gr{x}{w_0}$ is the same as reversing the columns of $M$, restricting to $\gr{e}{u}$, and then reversing the columns again. That is, letting $A$ denote the matrix with $1$'s on the antidiagonal and $0$'s elsewhere,
\[
    \res{M}{\gr{x}{w_0}}= [\res{(M A)}{\gr{e}{u}}]A.
\]

Also, note that reversing the columns of $\gr{v}{w_0}^{v_i}_i$ gives $\gr{e}{w}^{w_i}_i$. This means that
\begin{align*}
    \res{M}{\gr{v}{w_0}_{i}^{v_i}}&= [\res{(M A)}{\gr{e}{w}_i^{w_i}}]A.
\end{align*}

Thus, the proposition follows from taking determinants of both sides of \cref{eqn:deleteDotDetAnti} and applying \cref{prop:deleteDotDet}.
\end{proof}

 \section{Acknowledgements} 
 
 We would like to thank Pavlo Pylyavskyy for suggesting this problem to us, and for helpful discussions on Lusztig's dual canonical basis. We would also like to thank Alejandro Morales for pointing us towards \cite{Sjo} and Lauren Williams for helpful conversations.  The first author was partially supported by NSF RTG grant DMS-1745638. The second author was supported by NSF grant DGE-1752814.

\bibliographystyle{siam}
\bibliography{bibliography}

\end{document}